\documentclass{jnsao}
\usepackage[utf8]{inputenc}
\usepackage[finnish,english]{babel}
\usepackage[svgnames]{xcolor}
\usepackage[textsize=footnotesize,color=DarkOrange!40]{todonotes}
\usepackage{hyperref}
\usepackage{tikz, pgfplots}
\usepgfplotslibrary{colorbrewer}
\usepackage[inline]{enumitem}
\usepackage{algpseudocode}
\usepackage[nameinlink,capitalise]{cleveref}
\usepackage{nicematrix}
\usepackage{booktabs}
\usepackage{pgfplotstable}
\usepackage{changes-simple}

\makeatletter
\let\@the@H@page\relax
\makeatother

\definecolor{hrefcolor}{rgb}{0.0,0.4,0.7}
\definecolor{citecolor}{rgb}{0.0,0.35,0.2}
\definecolor{structure}{rgb}{0.09,0.09,0.44}
\hypersetup{
    colorlinks=true,
    linkcolor=citecolor,
    citecolor=citecolor,
    filecolor=hrefcolor,
    urlcolor=hrefcolor,
    pdfencoding=auto,
    hypertexnames=false,
}

\manuscriptcopyright{}
\manuscriptlicense{}
\manuscripteprinttype{arxiv}
\manuscripteprint{2605.26022}

\usepackage{float}

\usepackage{xcolor,cancel}

\floatstyle{ruled}
\floatname{algorithm}{Algorithm}
\newfloat{algorithm}{tbp!}{loa}
\numberwithin{algorithm}{section}

\tikzset{notestyleraw/.append style={align=justify}}

\usepackage{tikz}
\usepackage{pgfplots}
\usepgfplotslibrary{colorbrewer}

\def\ignorelegendentry#1{}
\pgfplotsset{
    every axis label/.append style = {font = \scriptsize},
    every tick label/.append style = {font = \scriptsize},
    ignore legend/.style={
        every axis legend/.code={}
    },
    log x ticks with fixed point/.style={
        xticklabel={
            \pgfkeys{/pgf/fpu=true}
            \pgfmathparse{exp(\tick)}%
            \pgfmathprintnumber[fixed relative, precision=3]{\pgfmathresult}
            \pgfkeys{/pgf/fpu=false}
        },
    },
    log y ticks with fixed point/.style={
        yticklabel={
            \pgfkeys{/pgf/fpu=true}
            \pgfmathparse{exp(\tick)}%
            \pgfmathprintnumber[fixed relative, precision=3]{\pgfmathresult}
            \pgfkeys{/pgf/fpu=false}
        },
    },
    compat=1.18,
    every axis/.append style={
        outer sep = 0pt,
        title style = {
            font = \normalsize
        },
        every x tick label/.append style={
          /pgf/number format/1000 sep={\ },
        },
        every y tick label/.append style={
          /pgf/number format/1000 sep={\ },
        }
    }
}

\usepackage{mdframed}

\mdfdefinestyle{examplebg}{
    backgroundcolor=black!6,%
    hidealllines=true,%
    innertopmargin=.3em,%
    innerbottommargin=.7em,%
    innerleftmargin=.7em,%
    innerrightmargin=.7em,%
    footnoteinside=false,
}

\surroundwithmdframed[style=examplebg]{example}

\author{%
    Jyrki Jauhiainen\thanks{%
        Department of Mathematics and Scientific Computing, University of Graz, Austria.
        \mbox{\email{jyrki.jauhiainen@uni-graz.at}},
        \orcid{0000-0001-6711-6997}
    }
    \and
    Yassine Nabou\thanks{%
        Department of Mathematics and Statistics, University of Helsinki, Finland.
        \mbox{\email{yassine.nabou@helsinki.fi}},
        \orcid{0009-0004-9805-8039}
    }
    \and
    Tuomo Valkonen\thanks{%
        MODEMAT Research Center in Mathematical Modeling and Optimization, Quito, Ecuador
        \emph{and}
        Department of Mathematics and Statistics, University of Helsinki, Finland.
        \email{tuomo.valkonen@iki.fi},
        \orcid{0000-0001-6683-3572}
    }
}

\shortauthor{J.~Jauhiainen, Y.~Nabou, T.~Valkonen}

\date{2026-05-25; revised 2026-10-16}

\acknowledgements{%
    This research has been supported by the Academy of Finland grant 345486.
    This research was funded in whole or in part by the Austrian Science
    Fund (FWF) \href{https://dx.doi.org/10.55776/F100800}{10.55776/F100800}.
}

\title{Dynamic inverse problems: Online regularisation theory}

\newcommand{\term}{\emph}

\newcommand{\field}[1]{\mathbb{#1}}
\newcommand{\N}{\mathbb{N}}

\newcommand{\R}{\field{R}}
\newcommand{\extR}{\overline \R}
\newcommand{\B}{B}

\newcommand{\norm}[1]{\|#1\|}

\newcommand{\inv}[1]{#1^{-1}}

\newcommand{\grad}{\nabla}
\newcommand{\freevar}{\,\boldsymbol\cdot\,}
\newcommand{\Union}\bigcup
\newcommand{\Isect}\bigcap
\newcommand{\union}\cup
\newcommand{\isect}\cap
\newcommand{\bigunion}\bigcup
\newcommand{\bigisect}\bigcap

\newcommand{\defeq}{:=}

\newcommand{\subdiff}{\partial}

\DeclareRobustCommand{\downto}{{{\mathchoice%
                    {\rotatebox[origin=c]{-20}{$\to$}}
                    {\rotatebox[origin=c]{-20}{$\to$}}
                    {\rotatebox[origin=c]{-20}{\scalebox{0.75}{$\to$}}}
                    {\rotatebox[origin=c]{-20}{\scalebox{0.6}{$\to$}}}
                }}}
\DeclareRobustCommand{\upto}{{{\mathchoice%
                    {\rotatebox[origin=c]{20}{$\to$}}
                    {\rotatebox[origin=c]{20}{$\to$}}
                    {\rotatebox[origin=c]{20}{\scalebox{0.75}{$\to$}}}
                    {\rotatebox[origin=c]{20}{\scalebox{0.6}{$\to$}}}
                }}}

\DeclareMathOperator*{\argmin}{arg\,min}

\DeclareMathOperator{\interior}{int}

\DeclareMathOperator{\closure}{cl}

\DeclareMathOperator{\Dom}{dom}

\DeclareMathOperator{\conv}{conv}

\DeclareMathOperator{\dist}{dist}
\DeclareMathOperator{\prox}{prox}

\newcommand{\iprod}[2]{\langle #1,#2\rangle}

\def\WOp{\Sigma^{-1/2}}
\def\EITmeas{\mathscr{I}}

\def \weaktostarSym{\setbox0=\hbox{$\rightharpoonup$}\rlap{\hbox
        to\wd0{\hss\raise1ex\hbox{$\scriptscriptstyle{*\,}$}\hss}}\box0}

\def\linear{\mathbb{L}}
\newcommand{\setto}{\rightrightarrows}

\def\extR{\overline \R}

\def\dualprod#1#2{\langle #1|#2\rangle}

\def\infconv{\mathop{\Box}}

\let\phi=\varphi
\let\epsilon=\varepsilon

\DeclareMathOperator{\Id}{Id}

\def\this#1{#1^k}

\let\opt\hat

\def\estgrad#1#2{\widetilde{\grad #1}_{#2}}
\def\estdiff#1#2{\widetilde{#1'_{\phantom{k\,}}}\!\!\!\!{\lower2pt\hbox{$_{#2}$}}}

\def\optx{\opt x}

\def\thisy{\this y}
\def\opty{\opt y}

\def \approxinSym{\setbox0=\hbox{$\in$}\rlap{\hbox
        to\wd0{\hss\raise1ex\hbox{$\sim$}\hss}}\box0}

\usepackage{ifthen}

\newcommand{\dis}[4][]{%
    d_{#2}%
    \ifthenelse{\equal{#1}{2}}{^2}{}%
    (#3, #4)%
}

\newcommand{\adaptdis}[4][]{%
    d_{#2}%
    \ifthenelse{\equal{#1}{2}}{^2}{}%
    \left(#3, #4\right)%
}

\newcommand{\trackingressum}[1][p]{\varsigma_p}
\newcommand{\tildetrackingressum}[1][p]{\tilde\varsigma_p}

\newcommand{\Urho}[2][\rho]{U_{#1}^{#2}}
\newcommand{\Ux}[2][\delta]{V_{#2,#1}}

\def\subdifffactor#1#2{\gamma_{#1,#2}}
\def\epsilonerror#1#2{\epsilon_{#1,#2}}
\def\boundconst#1#2{r_{#1,#2}}
\def\Rto#1{R_{:#1}}
\newcommand{\xto}[1]{x^{#1}}
\newcommand{\vto}[1]{v^{#1}}
\newcommand{\yto}[1]{y^{#1}}
\newcommand{\zto}[1]{z^{#1}}
\newcommand{\tildeyto}[1]{\tilde y^{#1}}
\newcommand{\tildeytoind}[2]{\tilde y^{#1}_{#2}}
\newcommand{\xtoind}[2]{x^{#1}_{#2}}
\newcommand{\ytoind}[2]{y^{#1}_{#2}}
\newcommand{\ztoind}[2]{z^{#1}_{#2}}
\newcommand{\Xto}[1]{X_{:#1}}
\newcommand{\Yto}[1]{Y_{:#1}}

\newcommand{\Kto}[1]{K_{:#1}}
\newcommand{\Ystarto}[1]{Y_{:#1}^*}
\newcommand{\barXto}[2]{\bar X_{:#1}^{#2}}
\newcommand{\optXto}[1]{\hat X_{:#1}}
\newcommand{\optxto}[1]{\hat x^{#1}}
\newcommand{\optxtostar}[1]{\hat x_{:#1}^*}

\newcommand{\xsolto}[2]{x^{#1}_{#2}}
\newcommand{\tildexsolto}[2]{\tilde x^{#1}_{#2}}
\newcommand{\optwto}[1]{\hat w^{#1}}
\newcommand{\Ato}[1]{A_{:#1}}
\newcommand{\bto}[2]{b^{#1}_{#2}}
\newcommand{\optbto}[1]{\hat b^{#1}}
\newcommand{\ellto}[1]{\ell_{:#1}}
\newcommand{\errto}[2]{e_{#1,#2}}
\newcommand{\Qto}[2]{Q_{:#1,#2}}
\newcommand{\Jto}[2]{J_{:#1,#2}}
\newcommand{\bregRto}[2]{B_{\Rto{N}}^{#2}}
\newcommand{\bregellto}[1]{B_{\ellto{N}}}
\def\DpredictConstrTo#1{\mathscr{Y}_{:#1}}
\newcommand{\Tto}[1]{T_{:#1}}
\newcommand{\ringGto}[1]{\mathring G_{:#1}}

\def\Npat{N_{\text{pat}}}
\def\Nelec{N_{\text{elec}}}

\let\MIN\underline
\let\MAX\overline

\begin{document}

\maketitle

\begin{abstract}
    We develop regularisation theory for dynamic inverse problems, solved using online methods with an infinite time horizon.
    Using concepts of subregularity to treat nonsmooth regularisers, we prove that time-averaged reconstruction errors converge to zero as noise, algorithmic errors, and regularisation vanish as the horizon grows. We illustrate the theory numerically with a dynamic electrical impedance tomography example.
\end{abstract}

\section{Introduction}
\label{sec:intro}

In dynamical inverse problems \cite{ip_special_issue_2028,holland2010reducing,hunt2014weighing,lipponen2011nonstationary,tuomov2024online-eit,alsaker2023multithreaded}, such as flow monitoring using electrical impedance tomography (EIT), one often considers a sequence of nonlinear inverse problems indexed by time. At each time step $k\in\N$, the parameter $x_k$ and the forward operator $A_k$ may change, and only noisy data $b_{k,\delta}$ are available. The underlying \term{infinite time horizon} operator equations take the form
\begin{equation}
    \label{eq:intro:infinite-time-horizon-operator-equation}
    (x_1, x_2\dots)\in\mathscr{X}
    \quad\text{and}\quad
    A_k(x_k) = \hat b_k
    \quad\text{for all}\quad k \ge 1,
\end{equation}
where the set $\mathscr{X}$ encodes a dynamic constraint.
It is practically impossible to solve this problem exactly: all of the data will never be available.
This naturally leads to the concept of an \term{online inverse problem}, where reconstructions must be updated sequentially as new data arrive.
To place this into context, consider first the classical \emph{static} inverse problem, where the operator and the data do not depend on time:
\[
    A(x) = \hat b,
\]
where again only noisy measurements $b_\delta$ of the exact data $\hat b$ are available.
A standard Tikhonov-type regularisation scheme reads
\begin{equation}
    \label{eq:static_problem}
    \min_{x\in X}
    J_{\delta}(x) + \alpha_\delta R(x),
    \qquad
    J_{\delta}(x) = \ell(A(x) - b_{\delta}),
\end{equation}
where $\ell$ is a data fidelity term and $R$ is a regularisation functional.
The central asymptotic question is whether minimisers $x_\delta$ converge, as $\delta \downto 0$, to an $R$-minimising solution of the exact equation, that is,
\[
    \hat x \in \argmin \{ R(x) \mid A(x) = \hat b \}.
\]
Classical convergence results rely on source conditions, Bregman divergences \cite{burger2004convergence,bredies2010total,artacho2013metric,ioffe2017variational,clasonvalkonen2020nonsmooth}, while we have in \cite{tuomov-subreg} exploited concepts of metric subregularity.
However, these theories are typically formulated for static problems with fixed operators and data.
In contrast, in the dynamic setting both the operators and the data evolve with time.
A natural, \emph{formal}, infinite time horizon minimum norm formulation is, therefore
\begin{equation}
    \label{eq:intro:infinite-time-horizon-minimal-norm}
    \min_{(x_1, x_2\dots)\in\mathscr{X}}~
    \sum_{j=1}^\infty R_j(x_j)
    \quad
    \text{subject to}
    \quad
    A_j(x_j)
    =
    \hat b_j
    \quad\text{for all}\quad
    j=1,2,\ldots.
\end{equation}
In the presence of noise, the corresponding Tikhonov-type problem becomes
\begin{equation}
    \label{eq:intro:infinite-time-horizon-tikhonov}
    \min_{(x_1, x_2\dots)\in\mathscr{X}}~
    \sum_{j=1}^\infty \ell_j\big(A_j(x_j) - b_{j,\delta}\big)
    +
    \sum_{j=1}^\infty \alpha_{j,\delta} R_j(x_j).
\end{equation}
This formulation highlights the structural differences between static and dynamic inverse problems and motivates the development of an online regularisation theory that extends classical convergence concepts to time-evolving operators and data.

The infinite time horizon gives the appropriate mindset for developing algorithms, as in \cite{online-tracking}, and relevant theory, as in this paper, for problems where we cannot wait for a long window of data, and then solve a dynamic-turned-static problem in the time window; where we want extremely fast real-time and low computational resource reconstructions in a long-running physical process, such as flow monitoring.
The temporal coupling encoded in $\mathscr{X}$, and approximated by predictors in dynamic online optimisation algorithms, permits a level of temporal super-resolution, so that one does not simply solve a sequence of static problems inexactly.
Indeed, as we will discuss in \cref{ex:reg:r,sec:infconv}, in such algorithms, $\mathscr{X}$ and the $R_j$ convolve into a spatiotemporal regulariser.

\paragraph{Contributions}

In \cref{sec:dynamic-regtheory}, we develop an \textbf{online regularisation theory} for dynamic inverse problems of the form \eqref{eq:intro:infinite-time-horizon-operator-equation}, where only noisy measurements $b_{k,\delta}$ of the exact data $\hat b_k$ are available.
Let  $\optxto{N}=(\hat x_0^N,\ldots,\hat x_N^N)$ solve an initial segment of length $N+1$ of \eqref{eq:intro:infinite-time-horizon-minimal-norm}, and let $\xsolto{N}{\delta}=(x_{\delta,0}^N,\ldots,x_{\delta,N}^N)$ \emph{approximately} solve the corresponding initial segment of \eqref{eq:intro:infinite-time-horizon-tikhonov}.
Denote the corresponding cumulative solution quality errors by  $e_{N,\delta}$ (as defined in detail in \cref{eq:reg:general:accuracy}).
Under mild assumptions on the temporal evolution of the data and of the forward operators $A_k$, we prove that the average
\[
    \frac{\norm{\optxto{N}-\xsolto{N}{\delta}}^2}{N+1} \to 0
    \quad\text{as all}
    \quad
    \begin{cases}
        \delta \downto 0,        & \text{noise reduces,}
        \\
        \alpha_\delta \downto 0, & \text{regularisation reduces},
        \\
        N \upto \infty,          & \text{time advances},
        \\
        \lim_{\delta \to 0} \lim_{N \to \infty}  \frac{e_{N,\delta}}{\alpha_\delta(N+1)} = 0,
                                 & \text{solution quality increases}.
    \end{cases}
\]
This extends classical static regularisation theory
\cite{engl1996regularization,burger2004convergence,valkonen2021regularisation,schuster2012regularization}
to the time-dependent setting.
Other authors have studied regularisation methods for dynamic inverse problems
\cite{sarnighausen2026regularization}, even with inexact operators
\cite{blanke2020inverse}.
Related online estimation approaches have also been studied in data
assimilation \cite{evensen2022data}, model predictive control
\cite{rawlings2017model}, and online parameter identification for
time-dependent partial differential equations \cite{boiger2016online}.
However, no works before ours appear to treat infinite time-horizon
problems and online regularisation methods.

Our approach is based on the same mathematical techniques that we used for the static regularisation theory of \cite{valkonen2021regularisation}. However, now we additionally have the time index $N \upto \infty$, to study asymptotic regularisation behaviour of infinite time-horizon problems. As in \cite{valkonen2021regularisation}, the solution quality errors $e_{N,\delta}$, allow
incorporating information on how practical \emph{online optimisation} methods \cite{tuomov-better-predict,tuomov-predict,tuomov2024online-eit} solve \eqref{eq:intro:infinite-time-horizon-tikhonov} inexactly.
Using such methods from our companion paper \cite{online-tracking} on online optimisation algorithms, in \cref{sec:numerical}, we numerically illustrate the above regularisation theory on an example from dynamic EIT.
In these methods, we would have $\xsolto{N}{\delta} = (x_{\delta,0},\ldots,x_{\delta,N})$ for all $N \in \N$, where $\{x_{\delta,k}\}_{k \in \N}$ is the corresponding sequence.
That is, the algorithms do not revisit history; they do not refine the solution for time index $k$ after first forming it.
Our theory applies to such methods, but equally to $\xsolto{N}{\delta}$ formed computationally more na\"ively by exactly solving initial segments of length $N$ of \eqref{eq:intro:infinite-time-horizon-tikhonov}.
Such approximate solutions refine the history as $N \upto \infty$.
Even in the static case, our results relax some of the assumptions of \cite{valkonen2021regularisation}.

Our work is dependent on regularity concepts of set-valued maps, such as metric subregularity, which we recall in the next \cref{sec:infconvularity}.
In \cref{sec:infconv}, we study the metric subregularity of spatiotemporal infimal convolutions, with the intent of providing ways of verifying the conditions of the theory for regularisation terms arising from the algorithms of \cite{tuomov2024online-eit,online-tracking}.

\paragraph{An issue of time indexing}

We generally use a superscript $N$ to denote a vector $\xto{N}$ that contains information until time index $N \in \N$.
To avoid confusion with exponentiation, and to contrast with the notation $A_k$ for a single time index $k$, we use the notation $\Ato{N}$ to denote quantities indexed up to time $N$.
In particular, $\xto{N} \in \Xto{N}$.
We will in typical applications have
\[
    x^N := (x_0^N,\dots,x_N^N) \in \Xto{N}
    \quad\text{with}\quad
    \Xto{N} = \prod_{k=0}^N X_k,
\]
but the general theory does not require this.
Even when this structure holds, there is not necessarily any correspondence between $\xtoind{N}{k}$ and $\xtoind{M}{k}$ for $0 \le k \le \min\{N,M\}$. (See the note about revisiting history above.)
When we use the notation $\Ato{N}$, and each $A_k$ for $k=0,\ldots,N$  does exist, such a relationship is, however, implied.

\paragraph{Further notation and elementary results}

We denote the extended reals by $\extR \defeq [-\infty,\infty]$.
We write $H: X \setto Y$ when $H$ is a set-valued map from the space $X$ to $Y$.
We write $\mathbb{L}(X;Y)$ for the space of bounded linear operators between the normed spaces $X$ and $Y$.
On a normed space $X$, for a point $x \in X$ and a set $U \subset X$, we write $\dist(x, U) \defeq \inf_{x' \in U} \norm{x-x'}_X$, where $\norm{\freevar}_X$ is the norm on $X$. We also write $\dist^2(x, U) \defeq \dist(x, U)^2$.
We write $\iprod{x}{x'}$ for the inner product between two elements $x$ and $x'$ of a Hilbert space $X$, and $\dualprod{x^*}{x} \defeq x^*(x)$ for the dual product or dual pairing in a normed space.
We write $\Id: X \to X$ for the identity operator on $X$ and $\delta_A: X \to \extR$ for the $\{0,\infty\}$-valued indicator function of a set $A \subset X$.

For Fréchet differentiable $F:X \to R$, we write $F'(x) \in X^*$ for the Fréchet derivative at $x \in X$. Here $X^*$ is the dual space to $X$.
For a convex function $F: X \to \extR$, we write $\subdiff F: X \setto X^*$ for its (set-valued) subdifferential map. The Fenchel conjugate is
\begin{equation}
    \label{eq:intro:fenchel-conjugate}
    F^* : X^* \to \extR,
    \qquad
    F^*(x^*) \defeq \sup_{x \in X}\left( \dualprod{x^*}{x}_{X^*,X} - F(x)\right).
\end{equation}
We write $\Dom F \defeq \{ x \in X \mid F(x) < \infty \}$ for the effective domain.

For $X$ a Hilbert space, we will frequently use Pythagoras' three-point identity
\begin{equation}
    \label{eq:intro:three-point}
    \iprod{x-y}{x-z}_X = \frac{1}{2}\norm{x-y}_X^2 - \frac{1}{2}\norm{y-z}_X^2 + \frac{1}{2}\norm{x-z}_X^2
    \quad
    (x,y,z\in X)
\end{equation}
and (inner product) Young's inequality
\begin{equation}
    \label{eq:intro:young}
    \iprod{x}{y}
    \le
    \norm{x}_X\norm{y}_X \le
    \frac{1}{2\alpha}\norm{x}_X^2 + \frac{\alpha}{2}\norm{y}_X^2
    \quad(x,y \in X,\, \alpha>0).
\end{equation}

\section{Metric subregularity and local subdifferentiability}
\label{sec:infconvularity}

We briefly recall the main notions of (strong) metric subregularity following \cite{artacho2008characterization,artacho2013metric,tuomov-subreg}.
A set-valued mapping $H: X \setto Y$ is said to be \emph{metrically subregular} at $\optx$ for $\opty$ if $\opty \in H(\optx)$ and there exist neighbourhoods $U \ni \optx$, $V \ni \opty$, and a constant $\kappa>0$ such that
\[
    \dist(x, \inv{H}(\opty)) \le \kappa \dist(\opty, H(x) \isect V)
    \quad\text{for all}\quad
    x \in U.
\]
If, in addition, $\optx$ is an isolated point of $\inv{H}(\opty)$ and the stronger inequality
\[
    \norm{x-\optx}_X \le \kappa \dist(\opty, H(x) \isect V)
    \quad\text{for all}\quad
    x \in U
\]
holds, then $H$ is said to be \emph{strongly metrically subregular} at $\optx$ for $\opty$.

For convex functionals $F: X \to \extR$, the metric and strong metric subregularity of the subdifferential mapping $\subdiff F$ can be equivalently characterised in terms of local quadratic growth conditions \cite{artacho2008characterization}. In particular,
\begin{align}
    \label{eq:subreg:local:strong:0}
    F(x) & \ge F(\optx) + \dualprod{\optx^*}{x-\optx} + \gamma \norm{x-\optx}_X^2
    \quad\text{for all}\quad
    x \in U
    \quad\text{and}
    \\
    \label{eq:subreg:local:semistrong:0}
    F(x) & \ge F(\optx) + \dualprod{\optx^*}{x-\optx} + \gamma \dist^2(x, \inv{[\subdiff F]}(\optx^*))
    \quad\text{for all}\quad
    x \in U,
\end{align}
correspond, respectively, to \emph{strong metric subregularity} and \emph{metric subregularity}
of $\subdiff F$ \emph{at $\optx$ for $\optx^* \in \subdiff F(\optx)$}, with the factor $\gamma>0$ in the neighbourhood $U$ of $\optx$.
The first condition~\eqref{eq:subreg:local:strong:0} expresses a localised version of strong subdifferentiability (equivalent to strong convexity in Hilbert spaces), ensuring isolated minimisers and quadratic stability.
By Taylor expansion, the inequality \eqref{eq:subreg:local:strong:0} holds even for nonconvex $F$ that is twice continuously differentiable.
The second condition \eqref{eq:subreg:local:semistrong:0} squeezes the set with the same subderivative---in particular, the set of minimisers---into a single point.

For use in our regularisation theory, we extend these concepts to be with respect to an arbitrary seminorm $\norm{\freevar}_\circ$ on $X$, and to incorporate an inexactness $\epsilon \ge 0$.
We also employ the concepts to nonconvex functions.
Thus, we call $F: X \to \extR$, $\norm{\freevar}_\circ$-\term{strongly locally $\epsilon$-subdifferentiable} at $\optx \in X$ for an $\optx^* \in X^*$, if there exists a $\gamma>0$ and a neighbourhood $U$ of $\optx$ such that
\begin{align}
    \label{eq:subreg:local:strong}
    F(x) & \ge F(\optx) + \dualprod{\optx^*}{x-\optx} + \gamma \norm{x-\optx}_\circ^2 - \epsilon
    \quad\text{for all}\quad
    x \in U.
    \\
    \intertext{It is $\norm{\freevar}_\circ$-\term{semi-strongly locally $\epsilon$-subdifferentiable} at $\optx \in X$ for $\optx^* \in X^*$ if}
    \label{eq:subreg:local:semistrong}
    F(x) & \ge F(\optx) + \dualprod{\optx^*}{x-\optx} + \gamma \dist_\circ^2(x, \inv{[\subdiff F]}(\optx^*)) - \epsilon
    \quad\text{for all}\quad
    x \in U,
\end{align}
where
\[
    \dist_\circ(x, \hat X) \defeq \inf_{\tilde x \in \hat X} \norm{x-\tilde x}_\circ
    \quad\text{for any}\quad \hat X \subset X.
\]
We will omit the mention of $\epsilon$ if it is zero.
In practise, as we will see in \cref{sec:infconv}, $\epsilon$ will be proportional to the noise level $\delta$.

We also introduce a variant more suitable for use with non-linear Fréchet differentiable forward-operator $A: X \to Y$.
Extending the definition of \cite{tuomov-subreg} to semi-norms, we then call $F: X \to \extR$ $(A, \norm{\freevar}_\circ)$-\emph{semi-strongly locally $\epsilon$-subdifferentiable} at $\optx \in X$ for an $\optx^* \in X^*$ with respect to a set $\opt X \subset X$ if there exist a neighbourhood $U \ni \optx$ and a factor $\gamma>0$ such that
\begin{equation}
    \label{eq:subreg:strong-subdiff}
    F(x) \ge F(\optx) + \dualprod{\optx^*}{x-\optx}
    + \gamma \norm{A'(\optx)(x -\optx)}_Y^2
    + \gamma \dist_\circ^2(x, \opt X)
    - \epsilon
    \quad (x \in U).
\end{equation}
This condition interpolates between \cref{eq:subreg:local:semistrong,eq:subreg:local:strong} in settings where $F$ may only exhibit growth in certain subspaces (captured by $A$) and stability with respect to a possibly non-singleton set of ground-truths $\opt X$.
This property plays a key role in establishing convergence and stability results for dynamic and online regularisation methods in the subsequent sections.\begin{remark}
    Throughout \cref{sec:dynamic-regtheory}, the semi-norm distance $(x, \optx) \to \norm{x-\optx}_\circ$ could be extended to an arbitrary non-negative function; its structure plays no particular role there.
    In \cref{sec:infconv}, we will minimally exploit the specific structure, but that section could also be extended to arbitrary functions if we were content with less explicit results.
\end{remark}

\section{Dynamic regularisation theory}
\label{sec:dynamic-regtheory}

We now develop a subregularity-based regularisation theory for dynamic inverse problems and their online solution methods.
Our proofs do not require the precise sum-of-data-fidelities-in-product-spaces structure of the introduction, so to reduce notational burden, we slightly relax the structure.

We assume that at each time step $N \in \N$, encompassing the entire history up to that time step, the uncorrupted cumulative data $\optbto{N} \in \Yto{N}$ is generated from a ground-truth $\optxto{N} \in \Xto{N}$ through the ideal model
\begin{equation}
    \label{eq:dynamic-regtheory:ideal-problem}
    \Ato{N}(\optxto{N}) = \optbto{N},
\end{equation}
where $\Xto{N}$ is a normed space, and $\Yto{N}$ a Hilbert space.
We require that the cumulative forward operator $\Ato{N} \in C^1(\Xto{N}; \Yto{N})$.
The true data is not available to us. Instead, we only observe a corrupted cumulative measurments
\[
    \bto{N}{\delta} \in \Yto{N}.
\]
Our goal, therefore, is not to solve \eqref{eq:dynamic-regtheory:ideal-problem} at each time step, but rather to track the unknown trajectory $\optxto{N}$ in an evolving, online fashion.
To this end, for the finite but advancing time horizon $N$, we consider inexact, sequentially constructed, solutions $\xsolto{N}{\delta}$ to the regularised cumulative reconstruction problem
\begin{subequations}
    \label{eq:reg:problem:regularised}
    \begin{equation}
        \label{eq:reg:problem:regularised:problem}
        \min_{\xto{N} \in \Xto{N}}
        \;\Jto{N}{\delta}(\xto{N})
        \;+\;
        \alpha_\delta\,\Rto{N}(\xto{N}),
    \end{equation}
    where $\Rto{N}: \Xto{N} \to \extR$ is a convex spatiotemporal regulariser, $\alpha_\delta>0$ a regularisation parameter corresponding to the corruption level $\delta>0$, and
    \begin{equation}
        \label{eq:reg:problem:regularised:j}
        \Jto{N}{\delta}(\xto{N})
        =
        \ellto{N}(\Ato{N}(\xto{N}) - \bto{N}{\delta})
    \end{equation}
    for some cumulative data fidelity function $\ellto{N}: \Yto{N} \to \R$.
    The regulariser $\Rto{N}$ can be used to incorporate the dynamic constraint $\mathscr{X}$ from \eqref{eq:intro:infinite-time-horizon-tikhonov} (up to time index $N$).
\end{subequations}

The next example explains how, in the typical case, the unknowns and observations encompass the whole history.

\begin{example}[Simple sequential observations]
    \label{ex:reg:simple}
    For a normed space $X$ of unknowns, and a Hilbert space $Y$ of measurements (often $\R^n$), let us be given  a data fidelity function $\ell: Y \to \R$, and, for each time index $k \in \N$, a forward operator $A_k \in C^1(X; Y)$, and corrupted measurements $b_{k,\delta}$ of $A_k \hat x_k = \hat b_k$.
    The time-dependence of the forward operator may model, for example, time-varying sampling patterns in magnetic resonance imaging, or potential excitation schemes in electrical impedance tomography.
    For
    \[
        \xto{N} =(\xtoind{N}{0},\ldots,\xtoind{N}{N}) \in \Xto{N} \defeq X^{N+1}
        \quad\text{and}\quad
        \yto{N} = (\ytoind{N}{0},\ldots,\ytoind{N}{N}) \in \Yto{N} \defeq Y^{N+1},
    \]
    we then set
    \[
        \Ato{N}(\xto{N}) = (A_0(\xtoind{N}{0}), \ldots, A_N(\xtoind{N}{N}))
        \quad\text{and}\quad
        \ellto{N}(\yto{N}) = \sum_{k=0}^N \ell(\ytoind{N}{k})
    \]
    as well as
    \[
        \bto{N}{\delta} = (b_{0,\delta},\ldots,b_{N,\delta)}
        \quad\text{and}\quad
        \optbto{N} = (\hat b_{0},\ldots,\hat b_{\delta}).
    \]
    Then \eqref{eq:reg:problem:regularised:j} gives the cumulative data fitting term until time index $N$ as
    \[
        \Jto{N}{\delta}(\xto{N})
        =
        \sum_{k=0}^{N} \ell_k(A_k(x_k) - b_{k,\delta}).
    \]
\end{example}

The regulariser $\Rto{N}$ will typically have a more complex structure:

\begin{example}[Spatiotemporal regularisers arising from online algorithms]
    \label{ex:reg:r}
    For some linear operators $K_k \in \linear(X_k; Y_k)$, ($k=0,\ldots,N$), write
    \[
        \Kto{N}\xto{N} \defeq (K_{0} \xtoind{N}{0}, \ldots, K_N \xtoind{N}{N}).
    \]
    Also let $T_k: Y^k \to \extR$ be convex, proper, and lower semicontinuous.
    In the online algorithms of \cite{tuomov2024online-eit,online-tracking}, the spatiotemporal regulariser $\Rto{N}$ arises as the specific instance\footnotemark[1]
    \[
        \Rto{N}(\xto{N})
        = \mathring T_{:N}(\Kto{N} \xto{N})
    \]
    where, for an initial slice $\DpredictConstrTo{N} \subset \Yto{N}$ of a \emph{dual} dynamic constraint,\footnotemark[2]
    \[
        \mathring T_{:N}(\yto{N})
        \defeq
        \sup_{\tildeyto{N} \in \DpredictConstrTo{N}}
        \sum_{k=1}^N
        \bigl[
            \iprod{\ytoind{N}{k}}{\tildeytoind{N}{k}} - T_k^*(\tildeytoind{N}{k})
            \bigr].
    \]
    If $\DpredictConstrTo{N}$ were convex, we would have \cite{tuomov-better-predict} the infimal convolution
    \[
        \mathring T_{:N}(\yto{N}) = [\Tto{N} \infconv \delta_{\DpredictConstrTo{N}}^*](\yto{N})
        \defeq \inf_{\tildeyto{N}} \left(
        \Tto{N}(\yto{N} - \tildeyto{N}) + \delta_{\DpredictConstrTo{N}}^*(\tildeyto{N})
        \right)
    \]
    where
    \[
        \Tto{N}(\yto{N}) \defeq \sum_{k=1}^{N} T_{k}(\inv\eta_k \thisy)
        \quad\text{and}\quad
        \delta_{\DpredictConstrTo{N}}^*(\yto{N}) = \sup_{\tildeyto{N} \in \DpredictConstrTo{N}} \iprod{\yto{N}}{\tildeyto{N}}.
    \]
    In typical applications of interest to us, each $T_k=\norm{\freevar}_{2,1}$ in a suitable finite element space, and $K_k=\grad$, so that $T_k \circ K_k$ models an isotropic total variation regulariser for the frame $k$.
    We recall that $\norm{\freevar}_{2,1}$ models the global $L^1$-norm of vector-valued functions with pointwise $2$-norms, and $B_{2,\infty}$ is the corresponding dual ball.
    Thus $\Rto{N}$ would be a spatiotemporal infimal convolution of the (dual) dynamic constraint, and the framewise total variation regularisers.
    When $\DpredictConstrTo{N}$ is not convex, we instead get a “sub-infimal” convolution: we have $\mathring T_{:N} \le \Tto{N} \infconv \delta_{\DpredictConstrTo{N}}^*$.
    We will in \cref{ex:subreg:r-improvement} show that equality actually holds here.
\end{example}

\footnotetext[1]{%
    To be more precise, in the algorithmic works \cite{tuomov2024online-eit,online-tracking}, where the corruption level $\delta$ plays no special role, one works instead $T_k$ with $G_k=\inv\alpha_\delta T_k$ that already incorporate the regularisation parameter.
    Then $G_k^*=\alpha_\delta T_k^*(\freevar/\alpha_\delta)$ \cite[Lemma 5.8]{clasonvalkonen2020nonsmooth}, and one defines $\ringGto{N}$ analogously to $\mathring T_{:N}$.
    For  $T_k=\norm{\freevar}_{2,1}$, in particular,  $G_k^*=\delta_{\alpha_\delta B_{2,\infty}}$.
    Hence the dual variables are scaled by $\alpha_\delta$.
    Correspondingly scaling the dual dynamic constraint $\DpredictConstrTo{N}$ by $\alpha_\delta$, it follows that $\ringGto{N} = \alpha_\delta \mathring T_{:N}$, so $\alpha_\delta\Rto{N}(\Xto{N}) = \ringGto{N}$.
}

\footnotetext[2]{In \cite[Corollary 2.13]{tuomov2024online-eit}, the definition of $\ringGto{N}$ has had the order of the sum and the supremum interchanged. The correct definition is found before \cite[Theorem 2.6]{tuomov-better-predict}, on which the former is based.}

We start with our general assumptions and constructions in \cref{sec:reg:general}.
We then provide a basic Bregman estimate with minimal assumptions in \cref{sec:reg:bregman}.
To provide norm-convergence to a specific minimum-$\Rto{N}$ solution of the noise-free problem, we consider in \cref{sec:reg:strong} a \term{strong source condition}.
We relax this in \cref{sec:reg:semistrong} to \term{semi-strong source condition} that gives convergence to the \emph{set of} minimum-$\Rto{N}$ solutions.

\subsection{General assumptions and concepts}
\label{sec:reg:general}

Let $N \in \N$. We denote by $\optXto{N} \subset \Xto{N}$ the set of $\Rto{N}$-minimising solutions
\begin{equation}
    \label{eq:reg:problem:truth:regularised}
    \optXto{N} \defeq \argmin_{x\in \Xto{N}}~ \Rto{N}(x)
    \quad \text{subject to} \quad \Ato{N} x=\optbto{N}
\end{equation}
We can derive necessary optimality conditions for this problem through the theory of Clarke subdifferentials \cite[Section 2.9]{clarke1990optimization}.
Since $\Ato{N} \in C^1(\Xto{N}; \Yto{N})$ and $\Rto{N}$ is convex and lower semicontinuous, both the functional $x \mapsto \Rto{N}(x)$ and the indicator function of the constraint set in \eqref{eq:reg:problem:truth:regularised} are regular in the sense of Clarke.\footnote[3]{That is, for locally Lipschitz functions the Clarke directional derivative $f^\circ$ agrees with the standard directional derivative $f'$, which is assumed to exist; see \cite[Definitions 2.3.4]{clarke1990optimization}. For infinite-valued functions, the Clarke tangent cone and the contingent cone of the epigraph have to agree \cite[Definitions 2.4.10]{clarke1990optimization}.}
Consequently, the sum and chain rules for the generalised gradient apply as equalities \cite[Thms.~2.9.8--2.9.9]{clarke1990optimization}.
The Fermat principle yields the desired optimality condition, known in the inverse problems literature as a source condition:\begin{assumption}[Basic source condition]
    \label{ass:reg:basic-sc}
    For given $N \ge 0$ and $\optxto{N} \in \optXto{N}$, there exists $\optwto{N} \in \Ystarto{N}$ such that
    \begin{equation}
        \label{eq:reg:basic-sc}
        \Ato{N}(\optxto{N}) =  \optbto{N}
        \quad\text{and}\quad
        - \Ato{N}'(\optxto{N})^*\optwto{N} \in \subdiff  \Rto{N}(\optxto{N}).
    \end{equation}
\end{assumption}

For corruption levels $\delta>0$, we also require inexact regularised solutions that satisfy certain quality guarantees:
\begin{assumption}[General assumptions on noisy data and approximate solutions]
    \label{ass:reg:general}
    For a given $N \ge 0$, let $\Xto{N}$ be a normed space and $\Yto{N}$ a Hilbert space, and the regularisation functional $\Rto{N}: \Xto{N} \to \extR$ be convex, proper, and lower semicontinuous. Moreover, let $\Ato{N} \in C^1(\Xto{N}; \Yto{N})$, and let $\ellto{N} \in C^1(\Yto{N})$.
    For all corruption levels $\delta \in (0, \delta^*)$ for some $\delta^*>0$, we then require:
    \begin{enumerate}[label=(\roman*)]
        \item \label{item:reg:general:noise}
              \textbf{Noise levels:}
              The corrupted measurements $\bto{N}{\delta} \in \Yto{N}$ of the ground-truth $\optbto{N} \in \Yto{N}$ satisfy for some fixed constants $C'>0$ and $q>0$
              \begin{equation}
                  \label{eq:reg:general:noise}
                  \frac{1}{N+1} \norm{\ellto{N}'(\bto{N}{\delta}-\optbto{N})}^2_{\Yto{N}^*} \le \delta
                  \quad\text{and}\quad
                  \frac{1}{N+1} \ellto{N}(\bto{N}{\delta}-\optbto{N}) \leq C' \delta^q.
              \end{equation}
        \item\label{item:reg:general:accuracy}
              \textbf{Accuracy of solutions of Tikhonov-type problems:}
              We are given $\xsolto{N}{\delta} \in \Xto{N}$ that, for some regularisation parameter $\alpha_\delta>0$, minimise
              \[
                  \Qto{N}{\delta}(\xto{N})
                  \defeq
                  \Jto{N}{\delta}(\xto{N}) + \alpha_\delta \Rto{N}(\xto{N})
                  \quad\text{where}\quad
                  \Jto{N}{\delta}(\xto{N}) \defeq \ellto{N}(\Ato{N}(\xto{N})-\bto{N}{\delta})
              \]
              to an accuracy $\errto{N}{\delta} \ge 0$ in the sense that, for some (and hence any) minimum-$\Rto{N}$ solution $\optxto{N} \in \optXto{N}$, we have
              \begin{equation}
                  \label{eq:reg:general:accuracy}
                  \Qto{N}{\delta}(\xsolto{N}{\delta})
                  \le
                  \Qto{N}{\delta}(\optxto{N})
                  +
                  \errto{N}{\delta}.
              \end{equation}
        \item \label{item:reg:general:quadratic}
              \textbf{Quadratic bounds:}
              For some $C>0$, we have
              \begin{equation}
                  \label{eq:reg:general:quadratic}
                  \ellto{N}(\Ato{N}(\xsolto{N}{\delta}) - \Ato{N}(\optxto{N})) \le C\left(\ellto{N}(\Ato{N}(\xsolto{N}{\delta}) - \bto{N}{\delta}) + \norm{\ellto{N}'(\bto{N}{\delta} - \optbto{N})}^2 \right).
              \end{equation}
    \end{enumerate}
\end{assumption}\begin{remark}[Quality of algorithmic solutions]
    In \eqref{eq:reg:general:accuracy}, we compare the (typically algorithm-constructed) approximate solution $\xsolto{N}{\delta}$ for the corruption level $\delta>0$ against a minimum-$\Rto{N}$ solution $\optxto{N}$ for corruption-free data. We could, instead, compare it against a minimiser  $\tildexsolto{N}{\delta}$ of $\Qto{N}{\delta}$, i.e., the best regularised solution for the problem with corruption level $\delta$ and finite time horizon $N$.
    However, the comparison against the ground-truth is, by definition, more permissive, and sufficient for our purposes.
\end{remark}\begin{remark}[Quadratic bound]
    The bound \eqref{eq:reg:general:quadratic} is variant of the pseudo-Hölder estimate of \cite[§5]{valkonen2021regularisation}.
    If $\ellto{N}=\frac{1}{2}\norm{\freevar}^2$, we have
    \[
        \begin{split}
            \ellto{N}(\Ato{N}(\xsolto{N}{\delta}) - \optbto{N}) - \ellto{N}(\Ato{N}(\xsolto{N}{\delta}) - \bto{N}{\delta})
             &
            =
            \iprod{\Ato{N}(\xsolto{N}{\delta})}{\bto{N}{\delta}-\optbto{N}}_{\Yto{N}}
            +
            \frac{1}{2}\norm{\optbto{N}}_{\Yto{N}}^2 - \frac{1}{2}\norm{\bto{N}{\delta}}_{\Yto{N}}^2
            \\
             &
            =
            \iprod{\Ato{N}(\xsolto{N}{\delta})_{\Yto{N}}-\bto{N}{\delta}}{\bto{N}{\delta}-\optbto{N}}_{\Yto{N}}
            +
            \frac{1}{2}\norm{\optbto{N}-\bto{N}{\delta}}_{\Yto{N}}^2
            \\
             &
            \le
            \ellto{N}(\Ato{N}(\xsolto{N}{\delta}) - \bto{N}{\delta}) + \norm{\optbto{N}-\bto{N}{\delta}}_{\Yto{N}}^2,
        \end{split}
    \]
    where in the last step we have used Young's inequality.
    This verifies \eqref{eq:reg:general:quadratic}.
\end{remark}\begin{remark}[Linearisation]
    Recall that  $\ellto{N}': \Yto{N} \to \Ystarto{N}$ is the Fr\'echet derivative of $\ellto{N}$, and $\bregellto{N}(y,y') \defeq \ellto{N}(y) - \ellto{N}(y') - \dualprod{\ellto{N}'(y')}{y-y'}$ is the Bregman divergence generated by $\ellto{N}$.
    The condition \cref{eq:reg:general:linearisation} is satisfied, for instance, when $\ellto{N}(z)=\frac{1}{2}\norm{z}_{\Yto{N}}^2$ is quadratic and the forward operator $\Ato{N}$ is linear (e.g., \cite{valkonen2021regularisation}).
    A static version of \cref{eq:reg:general:quadratic} has also been considered in \cite{valkonen2021regularisation} when $\ell$ is quadratic. Thus, our setting allows for more general data fidelity terms.
\end{remark}

\subsection{A basic Bregman estimate}
\label{sec:reg:bregman}

As a starting point for our analysis, we derive preliminary bounds on the distance of $\xsolto{N}{\delta}$ to $\optxto{N}$ in terms of Bregman divergences.
These results extend the work of \cite{valkonen2021regularisation,burger2004convergence} to the online setting.
Minding the basic source condition of \cref{ass:reg:basic-sc}, we define the Bregman divergence generated by $\Rto{N}$ at $\optxto{N}$ for $v \in\partial \Rto{N}(\optxto{N})$ by
\[
    \bregRto{N}{v}(\xsolto{N}{\delta}, \optxto{N}) := \Rto{N}(\xsolto{N}{\delta}) - \Rto{N}(\optxto{N}) + \dualprod{v}{x_{ \delta}^N - \optxto{N}}_{\Xto{N}^*,\Xto{N}}.
\]
We also define the Bregman divergence generated by $\ellto{N}$,
\[
    \bregellto{N}(y,\bar y)
    := \ellto{N}(y) - \ellto{N}(\bar y) - \dualprod{\ellto{N}'(\bar y)}{y-\bar y}_{\Yto{N}^*,\Yto{N}}.
\]
For the basic estimates of this subsection, in addition to \cref{ass:reg:basic-sc,ass:reg:general}, we require  the following approximate linearity assumption.\begin{assumption}[Approximate linearised data fidelity condition]
    \label{ass:reg:general:linearisation}
    There exist $\eta > 0$ such that
    \begin{equation}
        \label{eq:reg:general:linearisation}
        \begin{split}
             & \dualprod{\ellto{N}'(\optbto{N} - \bto{N}{\delta})}{\Ato{N}(\xsolto{N}{\delta}) - \Ato{N}(\optxto{N}) - \Ato{N}'(\optxto{N})(\xsolto{N}{\delta} - \optxto{N})}_{\Yto{N}}
            \\
             & \qquad + \bregellto{N}(\Ato{N}(\xsolto{N}{\delta}) - \bto{N}{\delta}, \Ato{N}(\optxto{N}) - \bto{N}{\delta})
            \geq \eta \norm{\Ato{N}'(\optxto{N})(\xsolto{N}{\delta} - \hat x_k^N)}_{\Yto{N}}^2
        \end{split}
    \end{equation}
    for all $N \in \N$ and $\delta \in (0, \delta^*)$.
\end{assumption}\begin{remark}[Approximate linearity]
    For $\ellto{N}(z)=\frac{1}{2}\norm{z}_{\Yto{N}}^2$, \eqref{eq:reg:general:linearisation} reads
    \begin{equation}
        \label{eq:reg:general:linearisation:quadratic}
        \begin{split}
             & \dualprod{\optbto{N} - \bto{N}{\delta}}{\Ato{N}(\xsolto{N}{\delta}) - \Ato{N}(\optxto{N}) - \Ato{N}'(\optxto{N})(\xsolto{N}{\delta} - \optxto{N})}_{\Yto{N}}
            \\
             & \qquad + \frac{1}{2}\norm{\Ato{N}(\xsolto{N}{\delta}) - \Ato{N}(\optxto{N})}_{\Yto{N}}^2
            \geq \eta \norm{\Ato{N}'(\optxto{N})(\xsolto{N}{\delta} - \hat x_k^N)}_{\Yto{N}}^2
            \quad
            \text{for all}
            \quad
            N \in \N.
        \end{split}
    \end{equation}
    This condition was used in \cite{valkonen2021regularisation} for conventional static regularisation.
    For the quadratic data fidelity, it follows that \cref{ass:reg:general:linearisation} holds, for instance, when the forward operator $\Ato{N}$ is linear.
\end{remark}

\begin{example}
    Suppose $\Ato{N}$ is twice continuously differentiable, and $\gamma M \le \Ato{N}'(\xto{N})^*\Ato{N}'(\xto{N}) \le C M$ for all $\xto{N} \in \Xto{N}$ and some $C,\gamma>0$ and a positive semi-definite operator $M$.
    Then, using the mean value theorem to write $\Ato{N}(\xsolto{N}{\delta}) - \Ato{N}(\optxto{N})=\Ato{N}'(\zeta)(\xsolto{N}{\delta} - \optxto{N})$ for some $\zeta \in [\xsolto{N}{\delta}, \optxto{N}]$, it is not difficult to see that \eqref{eq:reg:general:linearisation:quadratic} holds when $\optbto{N} - \bto{N}{\delta}$ is small, i.e., $\delta>0$ is small.
\end{example}

We now have the following bound.\begin{theorem}
    \label{thm:reg:bregman}
    Suppose that \cref{ass:reg:basic-sc,ass:reg:general,ass:reg:general:linearisation} hold for a $\optxto{N} \in \optXto{N}$.
    If $\delta \in (0, \delta^*)$, then
    \[
        0 \le\frac{1}{N+1}\bregRto{N}{\hat v}(\xsolto{N}{\delta} , \optxto{N})
        \leq
        \frac{\delta}{2\eta\alpha_\delta}
        +
        \frac{\errto{N}{\delta}}{\alpha_\delta(N+1)}
        + \frac{\alpha_\delta}{2\eta(N+1)}\norm{\optwto{N}}_{\Yto{N}}^2,
    \]
    where we write
    $
        \hat v \defeq - \Ato{N}'(\optxto{N}) \optwto{N}.
    $
\end{theorem}\begin{proof}
    For brevity, since $N$ is fixed throughout the proof, we drop the indexing by $N$, writing $\hat x \defeq \optxto{N}$, $\hat w \defeq \optwto{N}$, etc. By \cref{ass:reg:general} and the chosen notation, $A(\hat x) = \Ato{N}(\hat x) = \optbto{N} = \opt b$. The definition of the Bregman divergence then gives
    \[
        \begin{split}
            \Jto{N}{\delta}(x_\delta) - \Jto{N}{\delta}(\hat x )
             &
            =
            \ell(A(x_\delta) - b_\delta) - \ell(A(\hat x) - b_\delta)
            \\
             &
            =
            \iprod{\ell'(A(\hat x) - b_\delta)}{A(x_\delta) - A(\hat x)}
            + B_\ell(A(x_\delta) - b_\delta, A(\hat x) - b_\delta)
            \\
             &
            =
            \iprod{\ell'(\hat b - b_\delta)}{A(x_\delta ) - A(\hat x)} + B_\ell(A(x_\delta) - b_\delta, A(\hat x) - b_\delta)
            \\
            \MoveEqLeft[-2]
            -  \alpha_\delta\iprod{A'(\hat x)^* \hat w}{x_\delta - \hat x}
            +  \alpha_\delta\iprod{A'(\hat x)^*\hat w}{x_\delta - \hat x}
            \\
             &
            =
            \iprod{\ell'(\hat b - b_\delta)}{A(x_\delta ) - A(\hat x) -  A'(\hat x)(x_\delta - \hat x)}
            + B_\ell(A(x_\delta) - b_\delta , A(\hat x) - b_\delta)
            \\
            \MoveEqLeft[-1]
            + \iprod{\ell'(\hat b - b_\delta)}{A'(\hat x)(x_\delta - \hat x )}
            -  \alpha_\delta\iprod{A'(\hat x)^*\hat w}{x_\delta - \hat x} +  \alpha_\delta\iprod{A'(\hat x)^*\hat w}{x_\delta - \hat x}
            \\
             &
            \overset{\eqref{eq:reg:general:linearisation}}{\ge}
            \eta \norm{A'(\hat x)(x_\delta - \hat x)}^2
            + \langle  \ell'(\hat b - b_\delta) - \alpha_\delta\hat w , A'(\hat x)(x_\delta - \hat x ) \rangle
            \\
            \MoveEqLeft[-1]
            +  \alpha_\delta\iprod{A'(\hat x)^*\hat w}{x_\delta - \hat x}.
        \end{split}
    \]
    Continuing with Young's inequality and the triangle inequality, it follows that
    \[
        \begin{split}
            \Jto{N}{\delta}(x_\delta) - \Jto{N}{\delta}(\hat x )
             &
            \ge
            -\frac{1}{4\eta} \norm{  \ell'(\hat b - b_\delta) - \alpha_\delta \hat w}^2
            +  \alpha_\delta\iprod{A'(\hat x)^*\hat w}{x_\delta - \hat x}
            \\
             &
            \ge
            -\frac{1}{2\eta} \norm{\ell'(\hat b - b_\delta)}^2
            - \frac{\alpha_\delta^2}{2\eta}\norm{ \hat w}^2
            + \alpha_\delta\iprod{A'(\hat x)^*\hat w}{x_\delta - \hat x}
        \end{split}
    \]
    Using the definition of $\bregRto{N}{\hat v}$ and \cref{ass:reg:general}\,\cref{item:reg:general:noise,item:reg:general:accuracy}, it, therefore, follows that
    \[
        \begin{split}
            \errto{N}{\delta}
             &
            \ge
            -\frac{1}{2\eta} \norm{\ell'(\hat b - b_\delta)}^2
            - \frac{\alpha_\delta^2}{2\eta}\norm{ \hat w}^2
            + \alpha_\delta \left(\Rto{N}(x_{\delta}) - \Rto{N}(\hat x ) + \iprod{A'(\hat x)^*\hat w}{x_\delta - \hat x} \right)
            \\
             &
            \ge
            -\frac{\delta(N+1)}{2\eta} - \frac{\alpha_\delta^2}{2\eta}\norm{\hat w}_{\Yto{N}}^2
            + \alpha_\delta\bregRto{N}{\hat v}(x_{ \delta} , \hat x),
        \end{split}
    \]
    Dividing by $\alpha_\delta(N+1)$, the claim follows.
    The Bregman divergence $B^{\hat v}(x_{\delta}, \hat x)$ is non-negative by definition and to the source condition \eqref{eq:reg:basic-sc}.
\end{proof}

As a corollary, we get a convergence result for the average performance over long segments, as the corruption level $\delta \downto 0$, and the errors and the regularisation parameter behave appropriately.

\begin{corollary}
    Suppose that \cref{ass:reg:basic-sc,ass:reg:general,ass:reg:general:linearisation} hold for a $\optxto{N} \in \optXto{N}$ for all $N \in \N$.
    If, moreover,
    \begin{align*}
        &\lim_{\delta \downto 0} \alpha_{\delta} = 0,
        \qquad
        \lim_{\delta \downto 0} \delta/\alpha_{\delta} = 0,
        \qquad
        \lim_{\delta \downto 0} \lim_{N \to \infty}  \frac{e_{N,\delta}}{\alpha_\delta (N+1)} = 0,
        \\
        &\quad\text{and}\quad
        \sup_{N \in \N}
        \frac{1}{N+1}\norm{\optwto{N}}_{\Yto{N}}^2 < \infty,
    \end{align*}
    then,
    \[
        \lim_{\delta \to 0} \lim_{N\to \infty} \frac{1}{N+1} \bregRto{N}{\hat v}(\xsolto{N}{\delta} , \optxto{N}) = 0.
    \]
\end{corollary}

\subsection{Estimates based on a strong source condition}
\label{sec:reg:strong}

Bregman divergences may not be very informative; for example, the Bregman divergence of the absolute value function in $\R$ can only discern signs, but not magnitudes. We therefore next develop (semi-)norm-based estimates.
We begin with the following lemma, which will be used to show that the approximate regularised solutions $\xsolto{N}{\delta}$ remain close to $\optxto{N}$ when the noise level, regularisation parameter, and accuracy parameter are sufficiently small.\begin{lemma}
    \label{thm:reg:a-convergence}
    Suppose \cref{ass:reg:general} holds at a $\optxto{N} \in \optXto{N}$ for an $N\geq 0$. Then
    \[
        \frac{1}{N+1}\ellto{N}(\Ato{N}(\xsolto{N}{\delta}) - A_k(\hat x_k))
        \le
        \alpha_\delta C \boundconst{N}{\delta}
        \quad\text{and}\quad
        \frac{1}{N+1}\Rto{N}(\xsolto{N}{\delta}) \le \boundconst{N}{\delta}
    \]
    for
    \begin{equation}
        \label{eq:reg:boundconst}
        \boundconst{N}{\delta}
        \defeq
        \frac{C' \delta^q +\delta}{\alpha_{\delta}}
        + \frac{\errto{N}{\delta} }{\alpha_{\delta}(N+1)}
        + \frac{1}{N+1}\Rto{N}(\optxto{N}),
    \end{equation}
    which is independent of the choice of $\optxto{N} \in \optXto{N}$.
\end{lemma}\begin{proof}
    That $\boundconst{N}{\delta}$ does not depend on the choice of $\optxto{N} \in \optXto{N}$ follows from the definition of $\optXto{N}$ as the set of $\Rto{N}$-minimising solutions.

    By \cref{ass:reg:general}\,\cref{item:reg:general:noise,item:reg:general:accuracy,item:reg:general:quadratic}, and that $\Ato{N}(\optxto{N}) = \optbto{N}$ for all $k\in \{0, 1, \ldots, N\}$, we have
    \[
        \begin{aligned}
            C^{-1} & \ellto{N}(\Ato{N}(\xsolto{N}{\delta}) - \Ato{N}(\optxto{N}))
            + \alpha_\delta\Rto{N}(\xsolto{N}{\delta})
            \\
                   &
            \overset{\eqref{eq:reg:general:quadratic}}{\le}
            \ellto{N}(\Ato{N}(\xsolto{N}{\delta}) - \bto{N}{\delta}) + \norm{\ell'(\bto{N}{\delta} - \optbto{N})}_{\Yto{N}}^2
            +\alpha_\delta\Rto{N}(\xsolto{N}{\delta})
            \\
                   &
            \overset{\eqref{eq:reg:general:accuracy}}{\le}
            \errto{N}{\delta}
            + \norm{\ellto{N}'(\bto{N}{\delta}-\optbto{N})}_{\Yto{N}}^2
            +  \ellto{N}(\optbto{N} - \bto{N}{\delta}) + \alpha_\delta\Rto{N}(\optxto{N})
            \\
                   &
            \overset{\eqref{eq:reg:general:noise}}{\le}
            \errto{N}{\delta} + (N+1)(C' \delta^q + \delta) + \alpha_\delta\Rto{N}(\optxto{N}).
        \end{aligned}
    \]
    Hence, after dividing by $N+1$, the assertion follows from the elementary observation that for nonnegative quantities  $A,B,C$, the inequality $A+B\leq C$ implies $A\leq C$ and $B\leq C$.
\end{proof}

We will need the following “strong source condition” based on strong metric subregularity.
We refer to \cite{tuomov-subreg} for examples of the satisfaction of the corresponding static condition, in particular, with regard to total variation regularisation.
We discuss ways to prove the condition for the spatiotemporal regularisers of \cref{ex:reg:r} in \cref{sec:infconv}.\begin{assumption}[Strong source condition]
    \label{ass:reg:strong-sc}
    Let $N \in \N$.
    We say that $\optxto{N} \in \Xto{N}$ satisfies the \term{strong source condition} if $\Xto{N}$ is equipped with a semi-norm $\norm{\freevar}_{N,\circ}$, and the following hold:
    \begin{enumerate}[label=(\roman*)]
        \item
            \textbf{The basic source condition:}
            \eqref{eq:reg:basic-sc} holds for some for some $\optwto{N} \in \Ystarto{N}$.
        \item \label{item:reg:strong-sc:subdiff}
            \textbf{Strong local subdifferentiability:}
              There exists a $\delta^*>0$ such that for all $\delta \in (0, \delta^*)$, for given $\alpha_{\delta}$, the function $\Qto{N}{\delta}$ is \emph{\mbox{$\norm{\freevar}_{N,A,\circ}$}-strongly locally $\epsilonerror{N}{\delta}$-subdifferentiable} (see \cref{eq:subreg:local:strong}) at $\optxto{N}$ for $\optxtostar{N} \defeq \Jto{N}{\delta}'(\optxto{N}) - \alpha_\delta\Ato{N}'(\optxto{N})\optwto{N}$
              in a neighbourhood $\Ux{\optxto{N}}$ of $\optxto{N}$ with the factor $\subdifffactor{N}{\delta}>0$.
              Here the semi-norm
              \[
                  \norm{x^N}_{N,A,\circ}
                  \defeq
                  \sqrt{\norm{\Ato{N}'(\optxto{N}) x^N}_{\Yto{N}}^2 + \norm{x^N}_{N,\circ}^2}.
              \]

    \end{enumerate}
\end{assumption}

While not included as part of the source condition, to pass to the limit, we also require:
\begin{assumption}[Additional assumptions for convergence]
    \label{ass:reg:strong-misc}
    Let \cref{ass:reg:strong-sc} hold. We also require:
    \begin{enumerate}[label=(\roman*)]
        \item \label{item:strong-misc:nbd}
            \textbf{Sufficiently large neighbourhood of strong local subdifferentiability:}
              For some $\rho>0$, for all $\delta \in (0, \delta^*)$, we have
              \begin{equation*}
                  \Ux{\optxto{N}} \supset \Urho{N} \defeq \left\{ x^N \in \Xto{N} \middle|
                  \begin{array}{l}
                      \frac{1}{N+1}
                      \ellto{N}(\Ato{N}(\xto{N}) - \Ato{N}(\optxto{N})) \le \rho,
                      \\
                      \frac{1}{N+1} \Rto{N}(x^N)
                      \le \frac{1}{N+1} \Rto{N}(\optxto{N}) + \rho
                  \end{array}
                  \right\}.
              \end{equation*}
        \item  \label{item:strong-misc:boundconst-limit}
            \textbf{Asymptotically sufficient, yet vanishing, regularisation:}
              For $\boundconst{N}{\delta}$ defined in \eqref{eq:reg:boundconst}, we have $\lim_{\delta \downto 0} \sup_{N \in \N} \boundconst{N}{\delta} = \lim_{\delta \downto 0} \sup_{N \in \N} \frac{1}{N+1}\Rto{N}(\optxto{N}) < \infty$, along with $\lim_{\delta \downto 0} \alpha_\delta=0$
    \end{enumerate}
\end{assumption}\begin{remark}
    \Cref{ass:reg:strong-misc}\,\cref{item:strong-misc:nbd} is a somewhat strong requirement on the neighbourhood of strong local subdifferentiability of the single point $\optxto{N}$: it has to contain the level set neighbourhood $\Urho{N}$, which is independent of the choice of the point $\optxto{N} \in \optXto{N}$, by the definition of the latter set.
    We relax this to a covering property when treating the semi-strong source condition of \cref{sec:reg:semistrong}.

    The condition \cref{item:strong-misc:boundconst-limit} is a uniform-in-$N$ limiting property of the parameters and errors as $\delta\downto0$, as well a bound on the average values of the regularisation term for the minimal solutions.
    The $\sup_{N \in \N}$ can be relaxed to a $\lim_{N \to \infty}$ if we only require the lemma that follows the next example to hold for $N \ge N_\delta$ for some $N_\delta \in \N$.
\end{remark}

\begin{example}
    \label{ex:reg:strong-misc:containment}
    Suppose $\ellto{N}$ is minimised at zero.
    Let $\optxto{N} \in \optXto{N}$.
    Then, recalling the basic source condition \eqref{eq:reg:basic-sc}, the function $\Psi$ defined by
    \[
        \Psi(\xto{N}) \defeq \ellto{N}(\Ato{N}(\xto{N}) - \optbto{N}) + \Rto{N}(\xto{N}) + \dualprod{\Ato{N}'(\optxto{N})^*\optwto{N}}{\xto{N}-\optxto{N}}
    \]
    is minimised at $\optxto{N}$, and satisfies $\Psi(\optxto{N})=0$.
    (The last two convex terms are minimised at $\optxto{N}$ and, by assumption, so is the first term.)
    If $\Psi$ is strongly (globally) $\norm{\freevar}_{\Xto{N}}$-subdifferentiable at $\optxto{N}$ (for $0$), then $\Urho{N} \subset \B(\optxto{N},r)$ for any $r>0$ for small enough $\rho>0$.
    Consequently, \cref{ass:reg:strong-misc}\,\cref{item:strong-misc:nbd} holds.

    Note that $\Psi$ differs from $\Qto{N}{\delta}$, which has to be \mbox{$\norm{\freevar}_{N,A,\circ}$}-strongly locally $\epsilonerror{N}{\delta}$-subdifferentiable (\cref{ass:reg:strong-sc}\,\cref{item:reg:semistrong-sc:subdiff}) through the additional linear term, which does not effect the “strength” of the subdifferentiability, and through the true data $\optbto{N}$ in place of the noisy data $\bto{N}{\delta}$.
    The data can be swapped for small enough $\delta^*$ through Lipschitz assumptions on $\ellto{N}$, but the required growth here is possibly stronger due to requiring growth with respect to a norm instead of merely a semi-norm.
    However, $\Ux{\optxto{N}}$ may itself contain a subspace when the growth of  $\Qto{N}{\delta}$, is with respect to a proper semi-norm.
    Therefore, it may also be possible to relax the growth requirement on $\Psi$.
\end{example}

The next lemma explains why we require \cref{ass:reg:strong-misc}: it ensures that for small enough $\delta>0$, the algorithmic solutions $\xsolto{N}{\delta}$ belong to the neighbourhood of strong local subdifferentiability.\begin{lemma}
    \label{lemma:reg:containment}
    Suppose \cref{ass:reg:strong-misc} holds.
    Then there exists $\bar\delta \in (0, \delta^*)$ such that $\xsolto{N}{\delta} \in \Ux{\optxto{N}}$ for all $\delta \in (0, \bar\delta)$ and $N \in \N$.
\end{lemma}\begin{proof}
    By  \cref{ass:reg:strong-misc}\,\cref{item:strong-misc:boundconst-limit} and \cref{thm:reg:a-convergence}, there exists $\bar\delta \in (0, \delta^*)$, independent of $N$, such that $\xsolto{N}{\delta} \in \Urho{N}$ when $\delta \in (0, \bar\delta)$.
    By \cref{ass:reg:strong-misc}\,\cref{item:strong-misc:nbd}, it follows that $\xsolto{N}{\delta} \in \Ux{\optxto{N}}$.
\end{proof}

We can now obtain our main estimate for the strong source condition. We will afterwards convert it into a convergence result.\begin{theorem}
    \label{thm:reg:strong}
    Let $N \in \N$.
    Suppose \cref{ass:reg:general,ass:reg:strong-sc,ass:reg:strong-misc} hold at $\optxto{N} \in \Xto{N}$.
    Then there exists $\bar\delta \in (0, \delta^*)$, independent of $N$, such that if  $\delta \in (0, \bar\delta)$, we have
    \begin{equation}
        \label{eq:reg:strong:result}
        \frac{1}{N+1}\norm{\xsolto{N}{\delta} - \optxto{N} }^2_{N,\circ}
        \le
        \frac{1}{\subdifffactor{N}{\delta}}\left(
        \frac{\errto{N}{\delta}+\epsilonerror{N}{\delta}}{N+1}
        +\frac{\delta}{2\subdifffactor{N}{\delta}}
        +\frac{\alpha_{\delta}^2}{2\subdifffactor{N}{\delta}(N+1)}\norm{\optwto{N}}_{\Yto{N}}^2
        \right).
    \end{equation}
\end{theorem}\begin{proof}
    By \cref{lemma:reg:containment}, there exists $\bar\delta>\delta^*$ such that $\xsolto{N}{\delta} \in \Urho{N}$ for all $N \in \N$ when $\delta \in (0, \bar\delta)$.
    We also have
    \[
        \begin{split}
            \Jto{N}{\delta}(\optxto{N}) - \alpha_{\delta} \Ato{N}'(\optxto{N})^*\optwto{N}
             &
            =
            \Ato{N}'(\optxto{N})^*(\ellto{N}'(\Ato{N}(\optxto{N}) - \bto{N}{\delta}) - \alpha_{\delta} \optwto{N})
            \\
             &
            =
            \Ato{N}'(\optxto{N})^*(\ellto{N}'(\optbto{N} - \bto{N}{\delta}) - \alpha_{\delta} \optwto{N}).
        \end{split}
    \]
    Hence, by \cref{ass:reg:general}\,\cref{item:reg:general:accuracy} and \ref{ass:reg:strong-sc}\,\cref{item:reg:strong-sc:subdiff} followed by Young's inequality \eqref{eq:intro:young},
    \begin{equation}
        \label{eq:reg:strong:base-est}
        \begin{aligned}[t]
            \errto{N}{\delta} + \epsilonerror{N}{\delta}
             &
            \overset{\eqref{eq:reg:general:accuracy}}{\ge}
            \left[\Jto{N}{\delta}(\xsolto{N}{\delta}) + \alpha_\delta \Rto{N}(\xsolto{N}{\delta}) \right]
            -
            \left[\Jto{N}{\delta}(\optxto{N}) + \alpha_\delta\Rto{N}(\hat x ) \right]
            +
            \epsilonerror{N}{\delta}
            \\
             &
            \overset{\eqref{eq:subreg:local:strong}}{\ge}
            \dualprod{\Jto{N}{\delta}(\optxto{N}) - \alpha_{\delta}  \Ato{N}'(\optxto{N})^*\optwto{N}}{\xsolto{N}{\delta}-\optxto{N}}
            + \subdifffactor{N}{\delta} \norm{\xsolto{N}{\delta} - \optxto{N}}_{N,A,\circ}^2
            \\
             &
            =
            \iprod{
                \ellto{N}'(\optbto{N} - \bto{N}{\delta}) - \alpha_{\delta} \optwto{N}
            }{
                \Ato{N}'(\optxto{N})(\xsolto{N}{\delta} - \optxto{N})
            }
            \\
            \MoveEqLeft[-1]
            +
            \subdifffactor{N}{\delta} \norm{\Ato{N}'(\optxto{N})(\xsolto{N}{\delta} - \optxto{N})}_{Y_k}^2
            +
            \subdifffactor{N}{\delta} \norm{\xsolto{N}{\delta}-\optxto{N}}_{N,\circ}^2
            \\
             &
            \ge
            -
            \frac{1}{4\subdifffactor{N}{\delta}}\norm{\ellto{N}'(\optbto{N}-\bto{N}{\delta}) - \alpha_{\delta} \optwto{N}}_{\Yto{N}}^2
            +
            \subdifffactor{N}{\delta} \norm{\xsolto{N}{\delta}-\optxto{N}}_{N,\circ}^2.
        \end{aligned}
    \end{equation}
    Thus, rearranging and continuing with Young's inequality along with \cref{ass:reg:general}\,\cref{item:reg:general:noise}, we obtain
    \begin{equation}
        \label{eq:reg:strong:final-est}
        \begin{split}
            \norm{\xsolto{N}{\delta} - \optxto{N}}_{N,\circ}^2
             &
            \le
            \frac{1}{\subdifffactor{N}{\delta}}(\errto{N}{\delta} + \epsilonerror{N}{\delta})
            +\frac{1}{4\subdifffactor{N}{\delta}^2}\norm{\ellto{N}'(\optbto{N} - \bto{N}{\delta})-\alpha_{\delta} \optwto{N}}_{\Yto{N}}^2
            \\
             &
            \le
            \frac{1}{\subdifffactor{N}{\delta} }(\errto{N}{\delta} + \epsilonerror{N}{\delta})
            +\frac{\norm{\ellto{N}'(\optbto{N} - \bto{N}{\delta})}^2}{2\subdifffactor{N}{\delta}^2}
            +\frac{\alpha_{\delta}^2}{2\subdifffactor{N}{\delta}^2}\norm{\optwto{N}}_{\Yto{N}}^2
            \\
             &
            \le
            \frac{1}{\subdifffactor{N}{\delta} }(\errto{N}{\delta} + \epsilonerror{N}{\delta})
            +\frac{\delta(N+1)}{2\subdifffactor{N}{\delta}^2}
            +\frac{\alpha_{\delta}^2}{2\subdifffactor{N}{\delta}^2}\norm{\optwto{N}}_{\Yto{N}}^2.
        \end{split}
    \end{equation}
    Dividing by $N+1$, we obtain the claim.
\end{proof}

The next corollary shows norm convergence under similar parameter choices as in \cref{thm:reg:bregman}.
Define
\begin{subequations}
    \label{eq:reg:lim-quantities}
    \begin{align}
         & \mathscr{E}_{1,\delta}: = \lim_{N\to \infty} \frac{\errto{N}{\delta}+\epsilonerror{N}{\delta}}{(N+1)\subdifffactor{N}{\delta}} < \infty,
        \\
         & \mathscr{E}_{2,\delta}: = \lim_{N\to \infty}\frac{\delta}{\subdifffactor{N}{\delta}^2}< \infty,
        \\
         & \mathscr{E}_{3,\delta}: = \lim_{N\to \infty}\frac{\alpha_{\delta}^2}{(N+1)\subdifffactor{N}{\delta}^2}\norm{\optwto{N}}_{\Yto{N}}^2 < \infty.
    \end{align}
\end{subequations}

\begin{corollary}
    \label{cor:reg:strong-limit}
    Suppose \cref{ass:reg:general,ass:reg:strong-sc,ass:reg:strong-misc} hold at some $\optxto{N} \in \Xto{N}$ for all $N \in \N$.
    If
    \[
        \lim_{\delta \downto 0} \max\{\mathscr{E}_{1,\delta},\mathscr{E}_{2,\delta}, \mathscr{E}_{3,\delta}\} = 0,
    \]
    then
    \[
        \lim_{\delta \downto 0}  \lim_{N\to \infty} \frac{1}{N+1} \norm{\xsolto{N}{\delta}-\optxto{N}}^2_{N,\circ} = 0.
    \]
\end{corollary}\begin{proof}
    Since the threshold $\bar\delta$ of \cref{thm:reg:strong} is independent of $N$, it suffices to take $N \to \infty$ in \eqref{eq:reg:strong:result} divided by $N+1$ for all $\delta \in (0, \bar\delta)$, and then let $\delta \downto 0$.
\end{proof}

\subsection{Estimates based on a semi-strong source condition}
\label{sec:reg:semistrong}

We now reduce the assumption of \mbox{$\norm{\freevar}_{N,A,\circ}$}-strong local $\epsilonerror{N}{\delta}$-subdifferentiability, to \mbox{$(\Ato{N}, \norm{\freevar}_{N,\circ})$}-semi-strong local $\epsilonerror{N}{\delta}$-subdifferentiability.
This will yield convergence to a \emph{set} of minimal-$\Rto{N}$ solutions, instead of a specific minimal-$\Rto{N}$ solution.
We recall that the former set $\optXto{N}$ is defined in \cref{eq:reg:problem:truth:regularised}.
The results here, even in the static case with no $N$-dependence, relax the assumptions of \cite{valkonen2021regularisation}.\begin{assumption}[Semi-strong source condition]
    \label{ass:reg:semistrong-sc}
    Let $N \in \N$.
    We say that $\optxto{N} \in \Xto{N}$ satisfies the \term{semi-strong source condition}, if $\Xto{N}$ is equipped with a semi-norm $\norm{\freevar}_{N,\circ}$, and the following hold:
    \begin{enumerate}[label=(\roman*)]
        \item\label{item:reg:semistrong-sc:basic-sc}
            \textbf{The basic source condition:}
            \eqref{eq:reg:basic-sc} holds for some $\optwto{N} \in \Yto{N}$.
        \item\label{item:reg:semistrong-sc:subdiff}
            \textbf{Semi-strong local subdifferentiability:}
              There exist a $\delta^*>0$ such that for all $\delta \in (0, \delta^*)$, for given $\alpha_{\delta}>0$, the function $\Qto{N}{\delta}$ is \emph{$(\Ato{N}, \norm{\freevar}_{N,\circ})$-semi-strongly locally $\epsilonerror{N}{\delta}$-subdifferentiable} (see \eqref{eq:subreg:strong-subdiff}) at $\optxto{N}$ for $\optxtostar{N} \defeq \Jto{N}{\delta}'(\optxto{N}) - \alpha_{\delta} \Ato{N}'(\optxto{N})^*\optwto{N}$ with respect to $\optXto{N}$ in a neighbourhood $\Ux{\optxto{N}}$ of $\optxto{N}$ with the factor $\subdifffactor{N}{\delta}>0$.
    \end{enumerate}
\end{assumption}

Again, we refer to \cite{tuomov-subreg} for examples of the satisfaction of the corresponding static condition. We discuss ways to prove the condition for the spatiotemporal regularisers of \cref{ex:reg:r} in \cref{sec:infconv}.

To pass to the limit, we also require the following relaxation of \cref{ass:reg:strong-misc}.
\begin{assumption}[Additional assumptions for convergence]
    \label{ass:reg:semistrong-misc}
    We require:
    \begin{enumerate}[label=(\roman*)]
        \item\label{item:reg:semistrong-misc:containment}
            \textbf{Local cover by neighbourhoods of strong local subdifferentiability:}
              Let $\Urho{N}$ be defined in \cref{ass:reg:strong-misc}\,\cref{item:strong-misc:nbd}.
              Then, there exists a $\delta^*>0$ such that for all $\delta \in (0, \delta^*)$ and $N \in \N$, there exist collections $\barXto{N}{\delta} \subset \optXto{N}$ of points satisfying the \emph{semi-strong} source condition of \cref{ass:reg:semistrong-sc} for the same $\delta^*$ and $\subdifffactor{N}{\delta}$ with $\Union_{\bar x^N \in \barXto{N}{\delta}} \Ux{\bar x^N} \supset \Urho{N}$.
        \item\label{item:reg:semistrong-misc:boundconst-limit}
            \textbf{Asymptotically sufficient, yet vanishing, regularisation:}
              For $\boundconst{N}{\delta}$ defined in \eqref{eq:reg:boundconst},
              we have $\lim_{\delta \downto 0} \lim_{N \to \infty} \boundconst{N}{\delta} = \lim_{\delta \downto 0} \lim_{N \to \infty} \frac{1}{N+1}\Rto{N}(\optxto{N}) < \infty$, as well as $\lim_{\delta \downto 0} \alpha_\delta=0$
    \end{enumerate}
\end{assumption}

\begin{example}
    \label{ex:reg:semistrong-misc:containment}
    Recall \cref{ex:reg:strong-misc:containment}.
    If, there, we only have semi-strong subdifferentiability of $\Psi$, then $\Urho{N} \subset \optXto{N}+\B(0,r)$ for any $r>0$ for small enough $\rho>0$.
    Consequently, if there exists a set $\barXto{N}{\delta} \subset \optXto{N}$ where \cref{ass:reg:semistrong-sc} holds for uniform  $\subdifffactor{N}{\delta}$ and $\delta^*$ with $\Ux{\bar x^N} = \B(\optxto{N}, \epsilon_\delta)$  and $\barXto{N}{\delta} + \B(0, \epsilon_\delta) \supset \optXto{N}$ for some uniform $\epsilon_\delta>0$, then \cref{ass:reg:semistrong-misc}\,\cref{item:reg:semistrong-misc:containment} holds.
    The existence of the uniform $\epsilon_\delta>0$ holds, for instance, when $\barXto{N}{\delta}$ is finite, or when $\optXto{N}$ is compact and all $\optxto{N} \in \optXto{N}$ satisfy \cref{ass:reg:semistrong-sc}.
\end{example}

We denote the set distance with respect to $\norm{\freevar}_{N,\circ}$ by
\begin{equation}
    \label{eq:reg:semistrong-dist}
    \dist_{N,\circ}(\xto{N}, \optXto{N},) \defeq \inf_{\optxto{N} \in \optXto{N}} \norm{\optxto{N}- \xto{N}}_{N,\circ}
    \quad\text{for any}\quad
    \optXto{N} \subset \Xto{N}.
\end{equation}
Then we have the following initial estimate.\begin{lemma}
    \label{lemma:reg:semistrong-dist}
    Let $N \in \N$.
    Suppose \cref{ass:reg:general,ass:reg:semistrong-sc} hold at some $\optxto{N} \in \optXto{N}$.
    Let $\delta \in (0, \delta^*)$.
    If $\xsolto{N}{\delta} \in \Ux{\optXto{N}}$, then
    \[
        \dist_{N,\circ}^2(\xsolto{N}{\delta}, \optXto{N})
        \le
        \frac{\errto{N}{\delta}+\epsilonerror{N}{\delta}}{(N+1)\subdifffactor{N}{\delta}}
        + \frac{\delta}{2\subdifffactor{N}{\delta}2}
        + \frac{\alpha_{\delta}^2}{2(N+1)\subdifffactor{N}{\delta}^2} \norm{\optwto{N}}_{\Yto{N}}^2.
    \]
\end{lemma}\begin{proof}
    As in \eqref{eq:reg:strong:base-est}, using the assumed $(A_N,\norm{\freevar}_{N,\circ})$-strong local $\epsilonerror{N}{\delta}$-subdifferentiability of $\Qto{N}{\delta}$ (\cref{ass:reg:semistrong-sc}), we estimate
    \begin{equation*}
        \begin{aligned}[t]
            \errto{N}{\delta} + \epsilonerror{N}{\delta}
             &
            \overset{\eqref{eq:reg:general:accuracy}}{\ge}
            \left( \Jto{N}{\delta}(\xsolto{N}{\delta}) +\alpha_{\delta} \Rto{N}(\xsolto{N}{\delta}) \right)
            -
            \left(\Jto{N}{\delta}(\optxto{N})+\alpha_{\delta} \Rto{N}(\optxto{N}) \right)
            +
            \epsilonerror{N}{\delta}
            \\
             &
            \overset{\eqref{eq:subreg:strong-subdiff}}{\ge}
            \left( \dualprod{\Jto{N}{\delta}'(\optxto{N}) - \alpha_{\delta}  \Ato{N}'(\optxto{N})^*\optwto{N}}{\xsolto{N}{\delta} - \optxto{N}}_{\Xto{N}^*,\Xto{N}}
            +
            \subdifffactor{N}{\delta} \norm{\Ato{N}'(\optxto{N})(\xsolto{N}{\delta} - \optxto{N})}_{\Yto{N}}^2\right)
            \\
            \MoveEqLeft[-1]
            +
            \subdifffactor{N}{\delta} \dist_{N,\circ}^2(\xsolto{N}{\delta}, \optXto{N})
            \\
             &
            =
            \left(  \iprod{\ellto{N}'(\optbto{N}-\bto{N}{\delta})-\alpha_{\delta} \optwto{N}}{\Ato{N}'(\optxto{N})(\xsolto{N}{\delta}-\optxto{N})}_{\Yto{N}}
            +
            \subdifffactor{N}{\delta} \norm{\Ato{N}'(\optxto{N})(\xsolto{N}{\delta}-\optxto{N})}_{\Yto{N}}^2
            \right)
            \\
            \MoveEqLeft[-1]
            +
            \subdifffactor{N}{\delta} \dist_{N,\circ}^2(\xsolto{N}{\delta}, \optXto{N})
            \\
             &
            \ge
            -
            \frac{1}{4\subdifffactor{N}{\delta}}\norm{\ellto{N}'(\optbto{N}-\bto{N}{\delta})-\alpha_{\delta} \optwto{N}}_{\Yto{N}}^2
            +
            \subdifffactor{N}{\delta} \dist_{N,\circ}^2(\xsolto{N}{\delta}, \optXto{N}) .
        \end{aligned}
    \end{equation*}
    Now estimating as in \eqref{eq:reg:strong:final-est} yields the claim.
\end{proof}\begin{theorem}
    \label{thm:reg:semistrong-limit}
    Suppose \cref{ass:reg:general,ass:reg:semistrong-misc} hold.
    If, for $\mathscr{E}_{i,\delta}$ defined in \eqref{eq:reg:lim-quantities},
    \begin{equation}
        \label{eq:reg:semistrong-limit}
        \lim_{\delta \downto 0} \max\{\mathscr{E}_{1,\delta}, \mathscr{E}_{2,\delta}, \mathscr{E}_{3,\delta}\} = 0,
    \end{equation}
    then
    \[
        \lim_{\delta \downto 0} \lim_{N\to\infty} \frac{1}{N+1}\dist_{N,\circ}(\xsolto{N}{\delta}, \optXto{N})=0.
    \]
\end{theorem}

\begin{proof}
    Suppose, to reach a contradiction, that there exist
    $\varepsilon > 0$ and a sequence $\delta_j \downto 0$
    such that for every $j \in \N$ there exists a sequence
    $\{N_{j,k}\}_{k \in \N}$, $\lim_{k \to \infty} N_{j,k} = \infty$, satisfying
    \[
        \frac{1}{N_{j,k}+1}
        \dist_{N,\circ}(x_{\delta_j}^{N_{j,k}}, \hat X^{N_{j,k}})
        \ge \varepsilon.
    \]
    We may assume that $\delta_{j_0} \le \delta^*$ for $\delta^*$ given by \cref{ass:reg:semistrong-sc}\,\cref{item:reg:semistrong-sc:subdiff}.
    By \cref{thm:reg:a-convergence} and \cref{ass:reg:semistrong-misc}\,\cref{item:reg:semistrong-misc:boundconst-limit}, there exists $j_0 \in \N$ and for all $j \ge j_0$ an index $k(j) \in \N$ such that
    $
        x_{\delta_j}^{N_{j,k}} \in \Urho{N_{j,k}}
    $
    for $k \ge k(j)$.
    By \cref{ass:reg:semistrong-misc}\,\cref{item:reg:semistrong-misc:containment}, we have $x_{\delta_j}^{N_{j,k}} \in \Ux[\delta_j]{\bar x}$ for some $\bar x \in \barXto{N_{j,k}}{\delta_j}$, for all  $k \ge k(j)$ and $j \ge j_0$.
    By \cref{lemma:reg:semistrong-dist,eq:reg:semistrong-limit}, we then find some $n(j) \ge k(j)$ that satisfies the contradiction
    \[
        \lim_{j \to \infty}
        \frac{1}{N_{j,n(j)} + 1} \dist_{N,\circ}(\hat X^{N_{j,n(j)}}, x^{N_{j,n(j)}}_{\delta_j}) = 0.
        \qedhere
    \]
\end{proof}

\section{Strong local subdifferentiability and infimal convolutions}
\label{sec:infconv}

We now provide tools to verify the subregularity assumptions of \cref{ass:reg:strong-misc,ass:reg:semistrong-sc} when $\Rto{N}$ arises from the algorithms of \cite{tuomov2024online-eit,online-tracking} as discussed in \cref{ex:reg:r}.
In particular, we show in \cref{sec:infconv:basic} that $\Rto{N}$ is a proper infimal convolutions, not merely a “sub-infimal convolution”. There we also provide several results of general interest in the analysis of infimal convolutions.
Then in \cref{sec:infconv:subreg} we treat ways to obtain the required subregularity.
Throughout this section, $X$ is a normed space, unless otherwise indicated.

\subsection{Basic properties of infimal convolutions}
\label{sec:infconv:basic}

We recall that the infimal convolution $G \infconv H$ of two functions $G, H: X \to \extR$ is defined as
\[
    [G \infconv H](x) \defeq \inf_{\tilde x \in X}\left(
    G(\tilde x) + H(x-\tilde x)
    \right).
\]
In analogy with the Fenchel conjugate \eqref{eq:intro:fenchel-conjugate}, we also define for $E: X^* \to \extR$ the \term{Fenchel preconjugate} $E_*: X \to \extR$ through
\[
    E_*(x) \defeq \sup_{x^* \in X^*}\left( \dualprod{x^*}{x}_{X^*,X} - E(x^*)\right).
\]
Then the \term{biconjugate}
\[
    F^{**} \defeq (F^*)_* : X \to \extR.
\]
A well-known conjugate formula establishes $[G \infconv H]^* = G^* + H^*$; in finite dimensions, see, e.g., \cite[Theorem 11.23]{rockafellar-wets-va}.
It then follows that $[G \infconv H]^{**} = [G^* + H^*]_*$.
Now, if $G \infconv H$ is convex, proper, and lower semicontinuous, we have $[G \infconv H]^{**}=G \infconv H$ \cite[Theorem 5.2]{clasonvalkonen2020nonsmooth}, so obtain $G \infconv H = [G^* + H^*]_*$. It is easy to see that $G \infconv H$ is convex and proper if $G$ and $F$ are, however, the lower semicontinuity requires further assumptions.
See \cite[Theorem 11.23]{rockafellar-wets-va} in finite dimensions, and \cite[Proposition 12.14]{bauschke2017convex} in Hilbert spaces.
Since the result is not easily found in the literature in normed spaces, we include it here in full together with assumptions that are practical in our applications.\begin{lemma}
    \label{lemma:infconv:inverse-formula}
    Let $G, H: X \to \extR$ be convex and proper. Then $G \infconv H$ is convex and proper.

    If, moreover, $H$ or $G$ is Lipschitz continuous, then $G \infconv H$ is also lower semicontinuous; in consequence, $G \infconv H = [G^* + H^*]_*$.
\end{lemma}

\begin{proof}%
    By assumption, there exist some $\tilde x \in \Dom G$ and $x \in \Dom H$.
    Now, $x+\tilde x \in \Dom[G \infconv H]$, proving properness.
    Let then $\Psi(\tilde x, x) \defeq  G(\tilde x) + H(x-\tilde x)$ and $x_1,x_2\in X$ and $\lambda\in[0,1]$.
    The function $\Psi$ is convex, and there exist sequences $\{z^1_n\}_{n\in\N}$ and $\{z^2_n\}_{n\in\N}\subset X$ with $\Psi(x_1,z^1_n)\to[G \infconv H](x_1)$ and $\Psi(x_2,z^2_n)\to[G \infconv H](x_2)$.
    Convexity now follow from passing to the limit in
    \[
        [G \infconv H](\lambda x_1 + (1-\lambda)x_2)
        \le
        \Psi(\lambda (x_1,z^1_n) + (1-\lambda) (x_2,z^2_n) )
        \le
        \lambda \Psi(x_1,z^1_n) + (1-\lambda)\Psi(x_2,z^2_n).
    \]
    Finally, we need to show that $[G \infconv H](x) \le \liminf_{n \to \infty} [G \infconv H](x_n)$ for all converging sequences $x_n \to x$.
    Because $G \infconv H = H \infconv G$, we only need to consider the case that $H$ is Lipschitz continuous.
    Suppose, to reach a contradiction, that there exist $\epsilon \ge 0$ and $x_n \to x$ such that $[G \circ H](x) \ge [G \circ H](x_n) + \epsilon$ with $[G \circ H](x_n)<\infty$.
    By the definition of the infimal convolution, there then exists $z_n \in X$ such that $[G \circ H](x_n) \ge \Psi(x_n, z_n) - 1/n$.
    It follows that
    \[
        [G \circ H](x) \ge \Psi(x_n, z_n) + \epsilon -1/n = G(z_n) + H(x_n - z_n) + \epsilon -1/n.
    \]
    Since $H$ is $L$-Lipschitz for some $L$, it follows that
    \[
        [G \circ H](x)
        \ge
        G(z_n) + H(x - z_n) - L\norm{x-x_n} + \epsilon -1/n
        =
        \Psi(x, z_n) - L\norm{x-x_n} + \epsilon -1/n.
    \]
    Thus there exists $N \in \N$ such that for $n \ge N$, we have $[G \circ H](x) \ge \Psi(x, z_n) + \epsilon/2 \ge [G \circ H](x) + \epsilon/2$. This is the desired contradiction.
\end{proof}

Our first contribution in this section shows that the spatiotemporal “sub-infimal convolutions” of \cref{ex:reg:r} reduce to infimal convolutions with respect to the support function of the convex hull of the dual dynamic constraint set.\begin{lemma}
    \label{lemma:infconv:set-inverse-formula}
    Let $F=(G^* + \delta_U)_*$ for a convex, proper, and lower semicontinuous $G: X \to \extR$, and $\emptyset \ne U \subset X^*$.
    Then $F=[G \infconv \delta_{\closure \conv U}^*]^{**}$, where $\closure \conv U$ is the closed convex hull of $U$.

    If $G \infconv \delta_{\closure \conv U}^*$ is lower semicontinuous (e.g., the full conditions of \cref{lemma:infconv:inverse-formula} hold for $G$ and $H=\delta_{\closure \conv U}^*$) then $F=G \infconv \delta_{\closure \conv U}^*$.
\end{lemma}\begin{proof}
    The supremum of a lower semicontinuous concave function over a possibly nonconvex set equals its supremum over the closed convex hull of that set.
    Therefore,
    \[
        \begin{split}
            (G^* + \delta_U)^*(x)
             &
            =
            \sup_{x^* \in X^*} \left(
            \dualprod{x^*}{x}
            - G^*(x^*) - \delta_U(x^*)
            \right)
            \\
             &
            =
            \sup_{x^* \in X^*} \left(
            \dualprod{x^*}{x}
            - G^*(x^*) - \delta_{\closure \conv U}(x^*)
            \right)
            \\
             &
            =
            [G^* + \delta_{\closure \conv U}]_*(x)
            =
            [G \infconv \delta_{\closure \conv U}^*]^{**}(x).
        \end{split}
    \]
    By \cref{lemma:infconv:inverse-formula}, $G \infconv \delta_{\closure \conv U}^*$ is convex and proper.
    When it is also lower semicontinuous, it equals its biconjugate \cite[Theorem 5.2]{clasonvalkonen2020nonsmooth}.
\end{proof}

\begin{corollary}
    \label{cor:infconv:fullset-inverse-formula}
    Let $F=(\delta_B + \delta_U)_*$ for a non-empty $U \subset X^*$ and a closed, convex, non-empty $B \subset X^*$.
    Then $F=\delta_{B \isect \closure \conv U}^*$.
\end{corollary}\begin{proof}
    We have
    $
        \delta_B + \delta_{\closure \conv U} = \delta_{B \isect \closure \conv U},
    $
    which is lower semicontinuous by the closedness of $B \isect \closure \conv U$.
    The claim now follows from \cref{lemma:infconv:set-inverse-formula} with $G=(\delta_B)_*$.
\end{proof}

\begin{example}
    \label{ex:subreg:r-improvement}
    In \cref{ex:reg:r}, suppose the $G= \Tto{N}$ is Lipschitz continuous, and $H=\delta_{\closure \conv \DpredictConstrTo{N}}^*$.
    Then $\mathring T_{:N} = \Tto{N} \infconv \delta_{\closure \conv \DpredictConstrTo{N}}^*$,
    hence
    \[
        \Rto{N}(\xto{N}) =  [\Tto{N} \infconv \delta_{\closure \conv \DpredictConstrTo{N}}^*](\Kto{N} \xto{N}).
    \]
    In particular, with $Z$ a normed space, let
    \[
        T_{:N}: (Z^*)^N \to \R,
        \quad
        T_{:N}(\zto{N})=\sum_{k=0}^N \norm{\ztoind{N}{k}}_\bullet = (\delta_{B^{N+1}})_*(\zto{N})
    \]
    for some norm $\norm{\freevar}_\bullet=(\delta_B)_*$ in $Z^*$ with the corresponding dual unit ball $B$.
    Then
    \[
        \Rto{N}(\xto{N}) =  \delta_{B^{N+1} \isect \closure \conv \DpredictConstrTo{N}}^*(\Kto{N} \xto{N})
        =
        \sup_{\yto{N} \in B^{N+1} \isect \closure \conv \DpredictConstrTo{N}} \iprod{\yto{N}}{\Kto{N} \xto{N}}.
    \]
    If $\closure \conv \DpredictConstrTo{N} \supset \B^{N+1}$, we obtain $\Rto{N}(\xto{N})=\sum_{k=0}^N \norm{K_k \xtoind{N}{k}}_\bullet$.
    When we take $K_k=\grad$ to model total variation, this says that the spatiotemporal regulariser $\Rto{N}$ still penalises the gradients, but the dual temporal dynamics $\closure \conv \DpredictConstrTo{N}$ restrict the penalisation amount.
\end{example}

We also require a characterisation of the subdifferentials of the infimal convolution.
In Hilbert spaces, the following result can be found in \cite[Proposition 16.48]{bauschke2017convex}.\begin{lemma}
    \label{lemma:subreg:infconv-subdiff}
    Let $F=G \infconv H$ for a convex, proper, and lower semicontinuous $G,H:X \to \extR$.
    Then
    \[
        \begin{split}
            \subdiff F(x)
             &
            \supset
            \{
            x^* \in X^*
            \mid
            x^* \in \subdiff G(\tilde x) \isect \subdiff H(x-\tilde x)
            \text{ for some } \tilde x \in X
            \}
            \\
             &
            =
            \{
            x^* \in X^*
            \mid
            x^* \in \subdiff G(\tilde x) \isect \subdiff H(x-\tilde x),\,
            F(x) = G(\tilde x) + H(x-\tilde x)
            \text{ for some } \tilde x \in X
            \}.
        \end{split}
    \]
    Equality holds when there exists some $x_0^* \in \interior \Dom G^* \isect \Dom H^*$.
\end{lemma}\begin{proof}
    By the Fenchel–Young theorem \cite[Lemma 5.9]{clasonvalkonen2020nonsmooth}, $x^* \in \subdiff F(x)$ if and only if $x \in \subdiff F^*(x^*)$.
    We have $F^* = G^* + H^*$.
    By the sum rule for convex subdifferentials, $\subdiff F^*(x^*) \supset \subdiff G^*(x^*) + \subdiff H^*(x^*)$.
    Therefore $x^* \in \subdiff F(x)$ follows from $x \in \subdiff G^*(x^*) + \subdiff H^*(x^*)$, i.e., $\tilde x \in \subdiff G^*(x^*)$ and $x-\tilde x \in \subdiff H^*(x^*)$ for some $\tilde x$, i.e., $x^* \in \subdiff G(\tilde x) \isect \subdiff H(x-\tilde x)$.
    This proves the inclusion.

    The sum rule holds with equality when there exists some $x_0^* \in \interior \Dom G^* \isect \Dom H^*$. In that case the opposite implication also holds.
\end{proof}

It is not difficult to see that the infimal convolution of semi-norms is a semi-norm.
In fact, we can also square-and-root them:\begin{lemma}
    \label{lemma:infconv:seminorm}
    Let $G, H: X \to \extR$ be semi-norms.
    Write $G^2(x) \defeq G(x)^2$.
    Then $F(x) \defeq \sqrt{[G^2 \infconv H^2](x)}$ is a semi-norm.
\end{lemma}\begin{proof}
    The definition of the infimal convolution immediately establishes non-negativity.
    For any $x, z \in X$ and $\lambda>0$, since $F$ and $G$ are positively homogeneous and satisfy the triangle inequality,
    \[
        \begin{split}%
            [G^2 \infconv H^2]
            (x + \lambda z)
             &
            =
            \inf_{\tilde x, \tilde z \in X}\left(G(\tilde x + \lambda \tilde z )^2 + H(\tilde x + \lambda \tilde z - x - \lambda z)^2\right)
            \\
             &
            \le
            \inf_{\tilde x, \tilde z \in X}\left([G(\tilde x) + \lambda G(\tilde z)]^2 + [H(\tilde x -x) + \lambda H(\tilde z - z)]^2\right).
        \end{split}
    \]
    Hence, by the triangle inequality in $\R^2$ ($\norm{a+b}_2 \le \norm{a}_2 + \norm{b}_2$),
    \[
        \begin{split}
            \sqrt{[G^2 \infconv H^2](x + z)}
             &
            \le
            \inf_{\tilde x, \tilde z \in X}\left(
            \sqrt{G(\tilde x)^2 + H(\tilde x -x)^2} + \lambda \sqrt{G(\tilde z)^2 +  H(\tilde z - z)^2}\right)
            \\
             &
            = \sqrt{[G^2 \infconv H^2](x)} + \lambda \sqrt{[G^2 \infconv H^2](z)}.
        \end{split}
    \]
    This establishes both the triangle inequality and positive homogeneity.
\end{proof}

\subsection{Strong local subdifferentiability}
\label{sec:infconv:subreg}

Recall the concepts \cref{eq:subreg:local:strong,eq:subreg:local:semistrong} of strong and semi-strong local subdifferentiability.\begin{lemma}
    \label{lemma:infconv:subdiff-infconv}
    Let $F=G \infconv H$ for convex, proper, and lower semicontinuous $G,H:X \to \extR$ such that the full assumptions of \cref{lemma:infconv:inverse-formula} hold.
    Suppose
    \[
        S(x) \defeq \{ \tilde x \in X \mid \subdiff G(\tilde x) \isect \subdiff H(x-\tilde x) \ne \emptyset \}
    \]
    satisfies for a given $\tilde x \in S(\optx)$ and some $\epsilon,\rho>0$ the continuity property
    \begin{equation}
        \label{eq:infconv:metric:s-continuity}
        \dist(\tilde x, S(x)) \le \rho
        \quad\text{for all}\quad
        x \in \B(\optx, \epsilon).
    \end{equation}
    Let $\optx^* \in \subdiff F(\optx)$, and assume further that, for semi-norms $\norm{\freevar}_\diamond$ and $\norm{\freevar}_\triangle$ on $X$,
    \begin{enumerate}[label=(\roman*)]
        \item $G$ is $\norm{\freevar}_\diamond$-(semi)strongly locally subdifferentiable with the factor $\gamma>0$ at $\tilde x$ for $\optx^*$ in $B(\optx, \rho)$, and
        \item $H$ is $\norm{\freevar}_\triangle$-(semi)strongly locally subdifferentiable with the factor $\gamma$ at $\optx-\tilde x$ for $\optx^*$ in $B(\optx-\tilde x, \rho)$.
    \end{enumerate}
    Recalling \cref{lemma:infconv:seminorm}, let $\norm{\freevar}_\square \defeq \sqrt{\norm{\freevar}_\diamond^2 \infconv \norm{\freevar}_\triangle^2}$.
    Then $\optx^* \in \subdiff F(\optx)$ and $F$ is $\norm{\freevar}_\square$-(semi)strongly locally subdifferentiable at $\optx$ for $\optx^*$ with the factor $\gamma$ in $\B(\optx, \min\{\epsilon, 2\rho\})$.
\end{lemma}\begin{proof}
    We first treat strong local subdifferentiability.
    Let $x \in \B(\optx, \min\{\epsilon, 2\rho\})$.
    We have $[G \infconv H](\optx)=G(\tilde x) + H(\optx-\tilde x)$ for $\tilde x \in S(\optx)$.
    Also using the definition of the infimal convolution, it follows that
    \[
        \begin{split}
            F(x) - F(\optx) - \dualprod{\optx^*}{x-\optx}_{X^*,X}
             &
            =
            \inf_{x' \in X}\bigl(
            G(x') + H(x-x') - G(\tilde x) - H(\optx-\tilde x)
            \\
            \MoveEqLeft[-3]
            - \dualprod{\optx^*}{(x'-\tilde x)+(x-x'-(\optx-\tilde x))}_{X^*,X}
            \bigr).
        \end{split}
    \]
    Due to \eqref{eq:infconv:metric:s-continuity} and $x \in \B(\optx, \epsilon)$, there exists a $x' \in \B(\tilde x, \rho) \isect S(x)$, so we can use the assumed strong local metric subdifferentiabilities to obtain, as required
    \[
        \begin{split}
            F(x) - F(\optx) - \dualprod{\optx^*}{x-\optx}_{X^*,X}
             &
            \ge
            \inf_{x' \in \B(\tilde x, \rho)}\left(
            \gamma\norm{x'-\tilde x}_\diamond^2
            + \gamma\norm{(x-x')-(\optx-\tilde x)}_\triangle^2
            \right)
            \\
             &
            \ge
            \inf_{x' \in X}\left(
            \gamma\norm{x'-\tilde x}_\diamond^2
            + \gamma\norm{(x-\optx)-(x'-\tilde x)}_\triangle^2
            \right)
            =
            \gamma\norm{x-\optx}_\square^2.
        \end{split}
    \]

    For semi-strong local subdifferentiability, we similarly have
    \[
        \begin{split}
            F(x) - F(\optx) & - \dualprod{\optx^*}{x-\optx}
            \ge
            \inf_{x' \in X}\left(
            \gamma\dist^2(x', \inv{[\subdiff G]}(\optx^*))
            + \gamma\dist^2(x-x', \inv{[\subdiff H]}(\optx^*))
            \right)
            \\
                            &
            =
            \inf\Bigl\{
            \inf_{x' \in X}\left(
            \gamma\norm{x'-\tilde x}_\diamond^2
            + \gamma\norm{x-x' - \bar x}_\triangle^2
            \right)
            \Bigm|
            \optx^* \in \subdiff G(\tilde x) \isect \subdiff H(\bar x);\,
            \tilde x, \bar x \in X
            \Bigr\}
            \\
                            &
            =
            \inf\Bigl\{
            \inf_{x' \in X}\left(
            \gamma\norm{x'-\tilde x}_\diamond^2
            + \gamma\norm{x-\bar x - (x' - \tilde x)}_\triangle^2
            \right)
            \Bigm|
            \optx^* \in \subdiff G(\tilde x) \isect \subdiff H(\bar x - \tilde x);\,
            \tilde x, \bar x \in X
            \Bigr\}
            \\
                            &
            =
            \inf\{
            \gamma\norm{x- \bar x}_\square^2
            \mid
            \optx^* \in \subdiff G(\tilde x) \isect \subdiff H(\bar x - \tilde x);\,
            \tilde x, \bar x \in X
            \}.
        \end{split}
    \]
    On the other hand, \cref{lemma:subreg:infconv-subdiff} establishes that
    $\subdiff G(\tilde x) \isect \subdiff H(\bar x - \tilde x) \subset \subdiff F(\bar x)$.
    It follows that
    \[
        F(x) - F(\optx) - \dualprod{\optx^*}{x-\optx}
        \ge
        \inf\{
        \gamma\norm{x- \bar x}_\square^2
        \mid
        \optx^* \in \subdiff F(\bar x);\,
        \bar x \in X
        \}
        =
        \gamma\dist_\circ^2(x, \inv{[\subdiff F]}(\optx^*)),
    \]
    where $\dist_\circ$ is the set-distance generated by $\norm{\freevar}_\circ$.
\end{proof}\begin{remark}
    The continuity property \eqref{eq:infconv:metric:s-continuity} holds if $S$ has the Aubin property (see \cite[Chapter 27]{clasonvalkonen2020nonsmooth}) at $\optx$ for $\tilde x$ in $\B(\optx, \epsilon)$ with modulus $\kappa>0$.
    Indeed, then
    \[
        \dist(\tilde x, S(x)) \le \kappa \dist(\inv S(\tilde x), x) \le \kappa \norm{\optx -x} \le \kappa\epsilon=\rho
        \quad\text{for all}\quad
        x \in \B(\optx, \epsilon).
    \]
    By \cite[Theorem 28.3]{clasonvalkonen2020nonsmooth}, if $x^* \in \subdiff G(\tilde x) \isect \subdiff H(\optx-\tilde x)$, and
    \[
        0 \in D^*[\subdiff G](\tilde x|x^*)(x^{**}) + p
        \text{ and }
        p \in D^*[\subdiff H](\optx-\tilde x|x^*)(x^{**})
        \implies
        p=0, x^{**}=0
    \]
    then $S$ has the Aubin property at $x$ for $x^*$.
\end{remark}

\begin{example}[Growth properties of spatiotemporal total variation]
    \label{ex:infconv:spatiotemporal-tv}
    Let $\Rto{N}$ be formed from a spatial total variation regulariser as in \cref{ex:subreg:r-improvement,ex:reg:r}.
    Suppose
    \[
        G=T_{:N},
        \quad
        T_{:N}(\zto{N})=\sum_{k=0}^N \norm{\ztoind{N}{k}}_\bullet
        \quad\text{and}\quad
        H=\delta_{\closure \conv \DpredictConstrTo{N}}^*=(\delta_{B^{N+1}})_*(\zto{N})
    \]
    satisfy the assumptions of \cref{lemma:infconv:subdiff-infconv}.
    Then $T_{:N} \infconv \delta_{\closure \conv \DpredictConstrTo{N}}^*$ is $\norm{\freevar}_\square$-(semi-)strongly locally subdifferentiable; consequently $\Rto{N}(\xto{N}) =  [\Tto{N} \infconv \delta_{\closure \conv \DpredictConstrTo{N}}^*](\Kto{N} \xto{N})$ is  $\norm{\freevar}_{N,\circ}$-(semi-)strongly locally subdifferentiable for  $\norm{\xto{N}}_{N,\circ} \defeq \norm{\Kto{N}\xto{N}}_{\square}$.

    The semi-norm $\norm{\freevar}_\diamond$ may be zero on a very large subspace: it will generally only be non-zero when $z_k^N(\xi)=0$ for a spatial point $\xi$, and the corresponding subdifferential is in the interior of the dual ball; compare the analysis of both the Lasso and spatial total variation in \cite[Section 4]{valkonen2021regularisation}.
    Likewise, a small $\DpredictConstrTo{N}$ may produce $\norm{\freevar}_\square$ that provides growth only on a very small subspace.
    We, therefore, have to compensate for this weak growth estimate with the data term.
\end{example}

Suppose now that the basic source condition of \cref{ass:reg:basic-sc} holds for $\optxto{N}$.
To verify the strong or semi-strong source conditions \cref{ass:reg:semistrong-sc,ass:reg:strong-sc}, we need to verify the respective (semi-)strong local subdifferentiability of $\Qto{N}{\delta} = \Jto{N}{\delta} + \alpha_\delta \Rto{N}$ at $\optxto{N}$ for $\optxtostar{N} \defeq \Jto{N}{\delta}'(\optxto{N}) - \alpha_\delta\Ato{N}'(\optxto{N})\optwto{N}$.
Recall that
\[
    \Jto{N}{\delta}(x) = \ellto{N}(\Ato{N}(x)-\bto{N}{\delta}).
\]\begin{lemma}
    \label{lemma:infconv:j}
    Suppose \cref{ass:reg:general} holds, and that
    \begin{enumerate}[nosep,label=(\roman*)]
        \item\label{item:infconv:j:a}
              $\norm{[\Ato{N}'(\zeta)-\Ato{N}'(\hat x)]h} \le \beta\norm{\Ato{N}'(\hat x)h}$ for all $h$ and for all $\zeta \in U$ for some $\beta>0$ and a neighbourhood $U$ of $\optxto{N}$.
        \item $\ellto{N}$ is $\gamma$-strongly convex.
    \end{enumerate}
    Then, for any $\eta>0$ and $\xto{N} \in U$,
    \begin{equation}
        \label{eq:infconv:j}
        \Jto{N}{\delta}(\xto{N})-\Jto{N}{\delta}(\optxto{N})
        \ge
        \dualprod{\Jto{N}{\delta}'(\optxto{N})}{\xto{N}-\optxto{N}}
        -
        \delta(N+1)\inv\eta
        +
        \frac{\gamma-\eta\beta^2}{2}\norm{\Ato{N}'(\optxto{N})(\xto{N}-\optxto{N})}^2.
    \end{equation}
\end{lemma}\begin{proof}
    Using the $\gamma$-strong convexity of $\ellto{N}$ and the mean value theorem, for some $\zeta \in [\optxto{N},\xto{N}] \subset U$, we have
    \[
        \begin{split}
            \Jto{N}{\delta}(\xto{N})-\Jto{N}{\delta}(\optxto{N})
             &
            =
            \ellto{N}(\Ato{N}(\xto{N})-\bto{N}{\delta})
            -
            \ellto{N}(\Ato{N}(\optxto{N})-\bto{N}{\delta})
            \\
             &
            =
            \ellto{N}(\Ato{N}(\optxto{N})-\bto{N}{\delta}+\Ato{N}'(\zeta)(\xto{N}-\optxto{N}))
            -
            \ellto{N}(\Ato{N}(\optxto{N})-\bto{N}{\delta})
            \\
             &
            \ge
            \iprod{\grad\ellto{N}(\Ato{N}(\optxto{N})-\bto{N}{\delta})}{\Ato{N}'(\zeta)(\xto{N}-\optxto{N})}
            +
            \frac{\gamma}{2}\norm{\Ato{N}'(\zeta)(\xto{N}-\optxto{N})}^2
            \\
             &
            =
            \dualprod{\Jto{N}{\delta}'(\optxto{N})}{\xto{N}-\optxto{N}}
            +
            \iprod{\grad\ellto{N}(\optbto{N}-\bto{N}{\delta})}{[\Ato{N}'(\zeta)-\Ato{N}'(\optxto{N})](x-\optxto{N})}
            \\
            \MoveEqLeft[-1]
            +
            \frac{\gamma}{2}\norm{\Ato{N}'(\optxto{N})(\xto{N}-\optxto{N})}^2.
        \end{split}
    \]
    The claim now follows using Young's inequality, \cref{item:infconv:j:a}, and \cref{ass:reg:general}\,\cref{item:reg:general:noise}.
\end{proof}

Take $0<\eta<\gamma/\beta^2$.
Suppose $\Rto{N}$ is $\norm{\freevar}_{N,\circ}$-(semi-)strongly locally subdifferentiable with the factor $\gamma^R>0$ (independent of $N$) at $\optx$ for $- \alpha_\delta\Ato{N}'(\optxto{N})\optwto{N}$, that is
\begin{equation}
    \label{eq:infconv:r-semi}
    \Rto{N}(\xto{N})
    - \Rto{N}(\optxto{N})
    \ge
    - \dualprod{\alpha_\delta\Ato{N}'(\optxto{N})\optwto{N}}{\xto{N}-\optxto{N}}
    + \gamma^R\norm{\xto{N}-\optxto{N}}_{N,\circ}^2.
\end{equation}
or
\begin{equation}
    \label{eq:infconv:r-strong}
    \Rto{N}(\xto{N})
    - \Rto{N}(\optxto{N})
    \ge
    - \dualprod{\alpha_\delta\Ato{N}'(\optxto{N})\optwto{N}}{\xto{N}-\optxto{N}}
    + \gamma^R\dist_{N,\circ}^2(\xto{N}, \optXto{N})
\end{equation}
These estimates can in in applicable cases be verified with \cref{lemma:infconv:subdiff-infconv}; compare \cref{ex:infconv:spatiotemporal-tv}.
We can then by summing \eqref{eq:infconv:r-semi} (resp.~\eqref{eq:infconv:r-strong}) with \eqref{eq:infconv:j} obtain \cref{ass:reg:strong-sc} (resp.~\cref{ass:reg:semistrong-sc}) with
\[
    \epsilonerror{N}{\delta} =\delta(N+1)\inv\eta
    \quad\text{and, independent of $N$,}\quad
    \subdifffactor{N}{\delta} = \gamma \defeq \min\{\gamma-\eta\beta^2, \gamma^R\}/2.
\]
That is, for $\vto{N} \defeq \Jto{N}{\delta}'(\optxto{N}) - \alpha_\delta\Ato{N}'(\optxto{N})\optwto{N}$, we have in the strong case,
\begin{equation}
    \label{eq:infconv:q-semi}
    \begin{split}
        \Qto{N}{\delta}(\xto{N})
         &
        - \Qto{N}{\delta}(\optxto{N})
        \\
         &
        \ge
        \dualprod{\vto{N}}{\xto{N}-\optxto{N}}
        + \frac{\gamma-\eta\beta^2}{2}\norm{\Ato{N}'(\optxto{N})(\xto{N}-\optxto{N})}^2
        + \frac{\gamma^R}{2}\norm{\xto{N}-\optxto{N}}_{N,\circ}^2
        - \epsilonerror{N}{\delta}
        \\
         &
        \ge
        \dualprod{\vto{N}}{\xto{N}-\optxto{N}}
        + \subdifffactor{N}{\delta}\left(
        \norm{\Ato{N}'(\optxto{N})(\xto{N}-\optxto{N})}^2
        + \norm{\xto{N}-\optxto{N}}_{N,\circ}^2
        \right)
        - \epsilonerror{N}{\delta}
    \end{split}
\end{equation}
and, likewise, in the semi-strong case,
\begin{equation}
    \label{eq:infconv:q-strong}
    \begin{split}
        \Qto{N}{\delta}(\xto{N})
         &
        - \Qto{N}{\delta}(\optxto{N})
        \\
         &
        \ge
        \dualprod{\vto{N}}{\xto{N}-\optxto{N}}
        + \subdifffactor{N}{\delta}\norm{\Ato{N}'(\optxto{N})(\xto{N}-\optxto{N})}^2
        + \subdifffactor{N}{\delta}\dist_{N,\circ}^2(\xto{N}, \optXto{N})
        - \epsilonerror{N}{\delta}.
    \end{split}
\end{equation}
Recalling the discussion in \cref{ex:infconv:spatiotemporal-tv}, in the case of the strong source condition of \cref{ass:reg:strong-sc}, a part of $\norm{\Ato{N}'(\optxto{N})(\xto{N}-\optxto{N})}^2$ in \eqref{eq:infconv:q-semi} can be fused with $\norm{\xto{N}-\optxto{N}}_{N,\circ}^2$ to improve the latter semi-norm.
Indeed, compare the treatment of spatial total variation in \cite[Section 4.3]{valkonen2021regularisation}.
To do the same with \cref{eq:infconv:q-strong} and the semi-strong source condition of \cref{ass:reg:semistrong-sc} requires further assumptions.
In particular, if $\Ato{N}$ is linear then the value of
\[
    \Qto{N}{\delta}(\xto{N})
    - \Qto{N}{\delta}(\optxto{N})
    - \dualprod{\vto{N}}{\xto{N}-\optxto{N}}
\]
is independent of the choice of $\optxto{N} \in \optXto{N}$.
Therefore, in this case, we can take the supremum over $\optxto{N} \in \optXto{N}$ in \eqref{eq:infconv:q-strong}, and use $\sup A + \inf B \ge \inf(A+B)$ to improve the $\dist_{N,\circ}$.

By \cref{ex:reg:strong-misc:containment,ex:reg:semistrong-misc:containment}, the containment and covering properties of \cref{ass:reg:strong-misc}\,\cref{item:strong-misc:nbd} and \cref{ass:reg:semistrong-misc}\,\cref{item:reg:semistrong-misc:containment} can be studied using the very same techniques that we just applied to $\Qto{N}{\delta}$.

Recalling \eqref{eq:reg:lim-quantities}, we now get
\[
    \mathscr{E}_{1,\delta}
    = \lim_{N\to \infty}
    \frac{\errto{N}{\delta}}{(N+1)\gamma}
    +
    \frac{\delta}{\eta\gamma}.
\]
The condition $\lim_{\delta \downto 0} \mathscr{E}_{1,\delta}=0$ of \cref{thm:reg:semistrong-limit,cor:reg:strong-limit} can therefore be ensured by controlling the algorithm errors $\errto{N}{\delta}$.
The condition $\lim_{\delta \downto 0} \mathscr{E}_{2,\delta}=0$ automatically holds from the definition \eqref{eq:reg:lim-quantities}, while $\lim_{\delta \downto 0} \mathscr{E}_{3,\delta}=0$ requires $\norm{\optwto{N}}_{\Yto{N}}^2/(N+1)$ to be bounded and $\alpha_\delta \downto 0$.
To ensure the limiting properties of $\boundconst{N}{\delta}$ defined in \eqref{eq:reg:boundconst}, and as required by \cref{ass:reg:strong-misc}\,\cref{item:strong-misc:boundconst-limit}, we further require the relative rate control
\[
    \lim_{\delta \downto 0} \sup_{N \in \N}\left(
    \frac{\delta^q +\delta}{\alpha_{\delta}}
    + \frac{\errto{N}{\delta} }{\alpha_{\delta}(N+1)}
    \right)
    =0.
\]
For  \cref{ass:reg:semistrong-misc}\,\cref{item:reg:semistrong-misc:boundconst-limit}, the supremum becomes a limit.

\section{Numerical illustration}
\label{sec:numerical}

In this section, we numerically verify the regularisation theory of \cref{sec:dynamic-regtheory}  in the context of the dynamic EIT problem from the introduction.
In the latter, the goal is to reconstruct a time-varying conductivity distribution containing moving (and occasionally disappearing) inclusions. The evolution of the inclusions is governed by a transport equation corresponding to incompressible flow with constant density.

The numerical experiments, including the EIT forward model, data generation, discretisation choices, and noise model, closely follow the setup introduced in \cite{tuomov2024online-eit}.
The algorithmic setup follows our companion paper \cite{online-tracking} on online optimisation algorithms, which, itself, is extends the theory of \cite{tuomov2024tracking} on inexact algorithms for conventional PDE-constrained and bilevel optimisation.
We therefore only summarise the key ingredients here and refer the reader to those works for the full details.
Our Julia implementation of the experiments and algorithms is available on Zenodo \cite{jauhiainen2025online-eit-codes}.

\subsection{Mathematical model of EIT}
\label{sec:numerical:eit}

On time step $k$, the forward problem of EIT is governed by the Complete Electrode Model (CEM).
The objective is to reconstruct an electrical conductivity $x_k \in L^\infty(\Omega)$ inside the measurement domain $\Omega \subset \R^d$ from noisy measurements $\EITmeas_{\delta,j,k}$ of electrode currents $I_{j,k} = I(x_k, U_{j,k}) \in \R^{\Nelec}$ at $\Nelec \in \N$ electrodes over $j=1,\ldots,\Npat$ different potential patterns $U_{j,k} =(U_{j,k,1},\ldots,U_{j,k,\Nelec}) \in \R^{\Nelec}$ at the same electrodes.
We enforce enforce the uniform bounds $0 < \MIN x \le x_k \le \MAX x < \infty$ a.e.~in the imaging domain $\Omega$.
Then the physics are governed by the PDE
\begin{subequations}
    \label{eq:eit:cem}
    \begin{alignat}{4}
        \nabla\cdot(x_k\grad u_{j,k})                           & = 0          &  & \quad \text{in }\Omega,
        \\
        u^{j,k} + \zeta_i x_k\grad u_{j,k}\cdot\nu              & = U_{j,k,i}  &  & \quad \text{on }\partial\Omega_{e_i},
        \\
        x_k\grad u_{j,k}\cdot\nu                                & = 0          &  & \quad \text{on }\partial\Omega\setminus\cup_i \partial\Omega_{e_i},
        \\
        \int_{\partial\Omega_{e_i}}x_k\grad u_{j,k}\cdot\nu\,dS & = -I_{j,k,i} &  & \quad \text{for }i=1,\dots,\Nelec
        \text{ and }j=1,\ldots,\Npat.
    \end{alignat}
\end{subequations}
Here $\partial\Omega_{e_i}$ are the electrode surfaces, $\zeta_i>0$ are contact impedances, and $\nu$ is the outward unit normal.
This potential-to-current formulation reflects modern EIT measurement devices \cite{jauhiainen2020relaxed}.
In practice, only noisy measurements $\EITmeas_{\delta,j,k}$ of these currents are available, and the inverse problem is to recover $x_k$ from the measurements $\EITmeas_{\delta,j,k}$.
For a single time index $k$, this is commonly formulated as the total variation regularised optimisation problem
\begin{equation}
    \label{eq:numerical:static}
    \min_{x \in L^2(\Omega)} \underbrace{\frac12\sum_{j=1}^{\Npat}
    \norm{\WOp ( I(x,U_{j,k}) - \EITmeas_{\delta,j,k})}_2^2}_{E_{k+1}} + \delta_{[\MIN x,\MAX x]}(x) + \underbrace{\alpha_\delta\mathop{\mathrm{TV}}(x)}_{G_{k+1}\circ
     K_{k+1}}.
\end{equation}
where  $\WOp$ is a data precision matrix, and $\mathop{\mathrm{TV}}$ promotes structured sparsity.
Similar formulations have been studied in \cite{voss2018imaging,voss2019three,jauhiainen2021nonplanar}.
In the dynamic framework of \cref{sec:dynamic-regtheory}, following \cref{ex:reg:simple}, we would in the spaces $\Xto{N}=L^2(\Omega)^{N+1}$ and $\Yto{N}=(\R^{\Nelec \times \Npat})^{N+1}$ take
\begin{gather*}
    \ellto{N}(\yto{N}) = \sum_{k=0}^{N-1} \norm{\WOp\ytoind{N}{k}}_2^2,
    \quad
    \bto{N}{\delta} = (b_{0,\delta},\ldots,b_{N,\delta}),
    \quad
    b_{k,\delta}=(\EITmeas^{1,k}_\delta,\ldots,\EITmeas^{\Nelec,k}_\delta)
    \shortintertext{and}
    \Ato{N}(\xto{N}) = (A_0(\xtoind{N}{0}), \ldots, A_N(\xtoind{N}{N}))
    \quad\text{for}\quad
    A_k(\xtoind{N}{k}) =  (I(\xtoind{N}{k}, U_{1,k}), \ldots, I(\xtoind{N}{k}, U_{\Nelec,k})).
\end{gather*}
The spatiotemporal regulariser $\Rto{N}$ arises, as in \cref{ex:reg:r,ex:subreg:r-improvement}, as the temporal infimal convolution of the box constraints and total variation of \eqref{eq:numerical:static}, with the algorithm-dependent dual temporal dynamics.
The online algorithms of \cite{online-tracking,tuomov2024online-eit} satisfy $\xsolto{N}{\delta}=(x_{0,\delta}, \ldots, x_{N,\delta})$, i.e., the iterates for past time indices are not updated afterwards. Each $x_{k,\delta}$ is produced exclusively from $b_{k,\delta}$ and an initial prediction of its value formed using the past iterates, as we explain in \cref{sec:numerical:algorithm}.

\subsection{Experimental setup}
\label{sec:numerical:scenarios}

\def\illustration#1{
    \def\mesh{Disk16}
    \def\motion{#1}

    \def\choseniters{
        \raisebox{-.5\height}{\includegraphics[width=0.12\linewidth]{\pfx_1.png}}
        &
        \raisebox{-.5\height}{\includegraphics[width=0.12\linewidth]{\pfx_400.png}}
        &
        \raisebox{-.5\height}{\includegraphics[width=0.12\linewidth]{\pfx_800.png}}
        &
        \raisebox{-.5\height}{\includegraphics[width=0.12\linewidth]{\pfx_1200.png}}
        &
        \raisebox{-.5\height}{\includegraphics[width=0.12\linewidth]{\pfx_1600.png}}
        &
        \raisebox{-.5\height}{\includegraphics[width=0.12\linewidth]{\pfx_2000.png}}
    }\begin{figure}
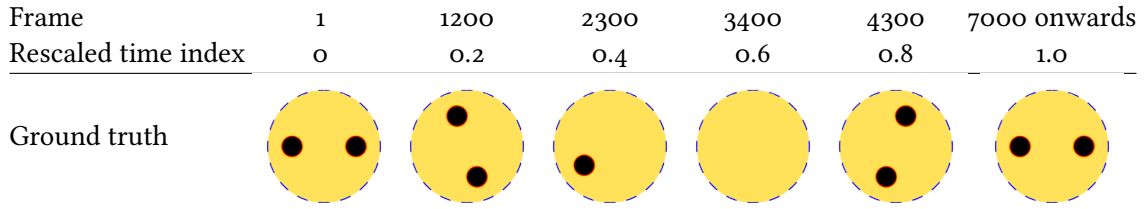

        \centering
        \setlength{\tabcolsep}{0pt}
        \begin{tabular}{lcccccc}
            Frame               & 1 & 1200 & 2300 & 3400 & 4300 & 7000 onwards

            \\
            Rescaled time index & 0 & 0.2  & 0.4  & 0.6  & 0.8  & 1.0

            \\
            \hline
            Ground truth
                                &
            \gdef\pfx{Results/Disk16/DisappearingMotion0.1/true/conductivity}%
            \choseniters%
        \end{tabular}
        \caption{Illustration of the experimental data.}
        \label{fig:illustration}
    \end{figure}

}

We perform our experiments on synthetic data in a disk-shaped domain $\Omega$ equipped with $\Nelec = 16$ equally spaced boundary electrodes.
The CEM is solved using a finite element approximation of the electrode potentials and the conductivity, both represented with piecewise linear basis functions.
For each time instance, current injection is applied sequentially at a single electrode while all remaining electrodes are grounded, resulting in $\Npat = 16$ excitation patterns. As is standard in EIT, the current measured at the excited electrode is excluded from the data, yielding $(\Nelec - 1)\Npat = 240$ measurements per frame.

For the movement of the inclusion, we limit our attention the following test scenario from \cite{tuomov2024online-eit}:
\begin{description}

    \item[Disappearing Inclusions]
          Two inclusions move at uniform speed along circular paths. One inclusion disappears at time index $1/4$ (near frame 1500) and the other at time index 2/4 (near frame 2800); both reappear at time index 3/4 (near frame 4200). This scenario violates the incompressibility assumption.
\end{description}
We make the following important changes:
\begin{enumerate}
    \item
          We use the same discretisation (5039 nodes) for both the forward simulations and the inverse problem. Although this would typically be regarded as an “inverse crime”, here it is a deliberate modelling choice: it eliminates discretisation-induced modelling error, allowing us to study the effect of measurement noise in isolation and, in particular, to consider the regime $\delta \downto 0$.
    \item
          We gradually slow down the originally uniform-speed motion of the inclusions for increasing time index $t = \min\{k/7000, 1\}$ by rescaling it to $s(t) \defeq t + t^2 - t^3$. This satisfies $s(0)=0$, $s(1)=1$, $s'(0)=1$, and $s'(1)=0$. Therefore, initially, the inclusions start at the same speed as in \cite{tuomov2024online-eit}, but---to allow evaluating the predictions of the regularisation theory---the movement slows to zero towards frame $t=7000$, after which, there is no movement.
          The total number of data frames is $N=10000$.
\end{enumerate}
We illustrate the ground-truth data in \cref{fig:illustration}.
To the ground-truth data, in our experiments, we add various levels of Gaussian noise $\delta$.

The regularisation parameter $\alpha_\delta$ is chosen heuristically. Based on a manual fit to empirically optimal values observed across a range of noise levels, we set $\alpha_\delta=1+\log_{10}(\delta)/10$.
This choice is intended only for the numerically relevant range of noise levels considered here and is not derived from theoretical principles; in particular, it is not suitable for $\delta \le 10^{-10}$.
All remaining experimental parameters, as well as the affine prediction scheme, are adopted from \cite{tuomov2024online-eit}.

\subsection{Algorithm and results}
\label{sec:numerical:algorithm}

    \begin{algorithm}
    \caption{Nonconvex predictive online primal-dual proximal splitting (POPD-N) \cite{tuomov2024online-eit,online-tracking}}
    \label{alg:pd:alg}
    \begin{algorithmic}[1]
        \Require For all $k\in\N$, on real Hilbert spaces $X_k$ and $Y_k$: convex, proper, lower semicontinuous $E_{k+1}\colon X_{k+1}\to\R$, $F_{k+1}\colon X_{k+1}\to\extR$, and $G_{k+1}^*\colon Y_{k+1}\to\extR$; a primal-dual predictor $P_k\colon X_k\times Y_k\to X_{k+1}\times Y_{k+1}$; and an operator $K_{k+1}\in\linear(X_{k+1};Y_{k+1})$.
        Estimates $\estdiff{E}{k+1}(\breve x_{k+1})$ of the gradients $E'_{k+1}(\breve x_{k+1})$.
        Step length parameters $\tau,\sigma >0$.
        \State Pick initial iterates $x^0\in X_0$ and $y^0\in Y_0$.
        \For{$k\in\N$}
        \State
        \label{item:alg:pd:predict}
        $(\breve x_{k+1},\breve y_{k+1}) \defeq P_k(x_k, y_k)$
        \Comment{prediction step}
        \State
        \label{item:alg:pd:primal}
        $x_{k+1} \defeq \prox_{\tau F_{k+1}}\!\bigl(\breve x_{k+1} - \tau \estgrad{E}{k+1}(\breve x_{k+1}) - \tau K_{k+1}^*\breve y_{k+1}\bigr)$
        \Comment{primal step}
        \State
        \label{item:alg:pd:dual}
        $y_{k+1} \defeq \prox_{\sigma G_{k+1}^*}\!\bigl(\breve y_{k+1} + \sigma K_{k+1}(2x_{k+1} - \breve x_{k+1})\bigr)$
        \Comment{dual step}
        \EndFor
    \end{algorithmic}
\end{algorithm}

For the numerical algorithm, we use the “exact” variant of the online primal-dual method of our companion paper \cite{online-tracking} with the step length and other parameters described therein, and the “affine scaling” dual predictor from \cite{tuomov-better-predict}.
This is produced here as \cref{alg:pd:alg}.
The variable indexing in the algorithm excludes the noise level $\delta$ for clarity, as it is fixed throughout each run of the algorithm.
In the “exact” variant that we use here, $\estgrad{E}{k+1}=\grad E_{k+1}$.
On each data frame $k$, the algorithm:
\begin{enumerate}
    \item Uses  optical flow  to predict the the primal iterate (conductivity) $x_k$ to form the prediction $\breve x_{k+1}$. This depends on iterates $x_k$ and $x_{k-1}$ \cite{tuomov2024online-eit}.
    \item Uses an “affine scaling” predictor (see \cite{tuomov2024online-eit}) to likewise form a prediction $\breve y_{k+1}$ from the dual iterate $y_k$. The dual iterates correspond to the variables that give rise to the total variation as a supremum over the dual ball; compare \cref{ex:subreg:r-improvement}.
    \item Solves both the PDE and its adjoint exactly to form $I(x_k, U_{j,k})$. (This is what the “exact” refers to. In \cite{online-tracking}, also inexact PDE solves are studied.)
    \item Takes \emph{only one} primal-dual proximal splitting step from the prediction $(\breve x_{k+1}, \breve y_{k+1})$ to form the next primal and dual iterates $(x_{k+1}, y_{k+1})$.
\end{enumerate}
The predictions, to transfer iterates from one time from to the next one, along with the inexact (single-step) solve of the frame-wise optimisation problem in the final step, introduce the errors $e_{k,\delta}$ in \eqref{eq:reg:general:accuracy}.

Using the above experimental setup, \cref{fig:reg} illustrates how the reconstruction quality increases as the noise level $\delta \downto 0$, and the time index $N \upto \infty$.
This confirms the claim of \cref{thm:reg:strong} on the convergence of the algorithmic solutions to a (minimum norm) ground truth as the noise level $\delta \downto 0$, given that the errors and regularisation parameter have appropriate limiting behaviour.

\illustration{DisappearingMotion}

\begingroup

\def\doplot#1{
    \def\motion{#1}
    \begin{tikzpicture}
        \begin{axis}[
                width=0.75\linewidth,
                height=0.35\linewidth,
                grid=major,
                xlabel={Frame},
                ylabel={$\frac{1}{N+1} \sum_{t=0}^N \norm{\hat x^N_t- x_t}^2$},
                ymode=log,
                scaled x ticks=false,
                xminorticks=true,
                yminorticks=true,
                axis x line*=bottom,
                axis y line*=left,
                legend style={
                        legend pos=north east,
                        inner sep=1pt,
                        outer sep=1pt,
                        draw=none,
                        fill=none,
                        font=\scriptsize,
                        at={(1.02,1)},
                        anchor=north west,
                    },
                line width = 1pt,
                colormap/PuBu,
                cycle list = {[colors of colormap={400,550,...,1000}]},
                log ticks with fixed point,
                ytick={10, 5, 1, 0.5, 0.1, 0.05, 0.01},
                ymax=0.5,
                xmin=1,
                xmax=10000,
            ]

            \foreach \del in {0.100, 0.050, 0.010, 0.005}{ 
                    \addplot
                    table [x=iter, y=gt_cum_avg_sqerror, col sep=comma]
                        {RegResults/\mesh/\motion_noise\del000_logalpha_polyspeed7000_lowres/exact\extra_\pred_\steplengths/timings.txt};
                    \addlegendentryexpanded{$\delta = \del$}
                }

        \end{axis}
    \end{tikzpicture}
}

\def\plotall#1{
    \gdef\pred{#1}

    \begin{figure}
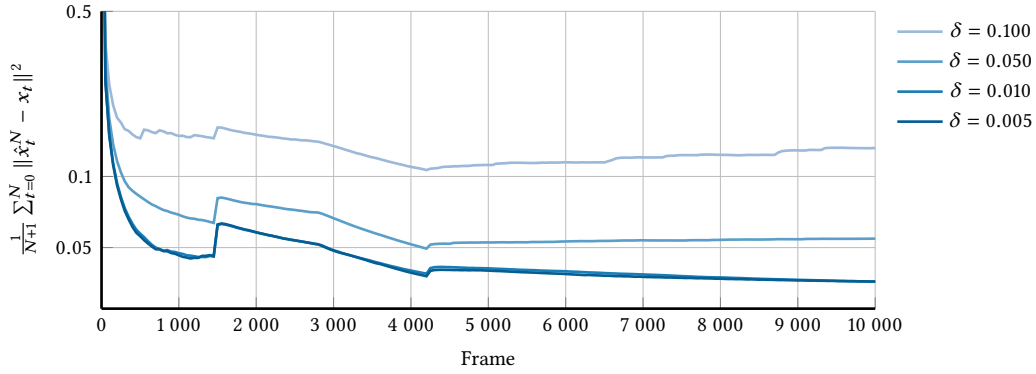

        %
        \centering
        \doplot{DisappearingMotion}%
        \caption{Illustration of asymptotic regularisation behaviour for various noise levels $\delta$ and regularisation parameter $\alpha=1 + \log_{10}(\delta)/10$.}
        \label{fig:reg}
    \end{figure}

}

\def\mesh{Disk16}
\def\steplengths{sigma10.0_tau0.0053}
\def\extra{_fixedstep}

\plotall{PDP2}

\endgroup

\section{Conclusion}
\label{sec:conclusion}

In this paper, we have developed a regularisation theory for dynamic inverse problems solved by online methods over an infinite time horizon. Using metric subregularity and related local regularity concepts of nonsmooth and set-valued functions, we established time-averaged convergence of reconstructions under suitable conditions on the noise, regularisation, and algorithmic errors.
The theory also accommodates the spatiotemporal regularisers arising from online optimisation methods. We illustrated the results numerically on a dynamic electrical impedance tomography problem using an online primal-dual algorithm.

\bibliographystyle{jnsao}
\end{document}